\documentclass[11pt, reqno]{amsart}
\usepackage{subfig, enumerate, tikz-cd}
\usepackage{bookmark}
\usepackage{filecontents}
\usepackage[shortlabels]{enumitem}
\usepackage[section] {placeins}
\usepackage{cases}
\usepackage{amscd,amsfonts,amssymb,amsmath,tikz-cd}
\usepackage{bm}
\usepackage{hyperref}
\usepackage{epsfig}
\usepackage{mathtools}
\usepackage{color}
\usepackage[normalem]{ulem}
\usepackage[sort,nocompress]{cite}
\newtheorem{theorem}{Theorem}[section]

\newtheorem{proposition}[theorem]{Proposition}
\newtheorem{corollary}[theorem]{Corollary}
\newtheorem{conjecture}[theorem]{Conjecture}
\newtheorem{lemma}[theorem]{Lemma}
\newtheorem{question}[theorem]{Question}

\newtheorem*{Theorem}{Theorem}
\newtheorem*{Problem}{Problem}
\newtheorem*{claim}{Claim}
\newtheorem*{questions}{Questions}
\theoremstyle{definition}
\newtheorem{definition}{Definition}[section]

\newtheorem*{examples}{Examples}
\newtheorem*{notation}{Notation}
\newtheorem{remark}{Remark}[section]
\newtheorem*{ack}{Acknowledgement}
\DeclareMathOperator{\tr}{tr}
\DeclareMathOperator{\Cl}{Cl}
\DeclareMathOperator{\type}{Type}
\DeclareMathOperator{\M}{M}
\DeclareMathOperator{\Qnd}{Qnd}
\DeclareMathOperator{\Grp}{Grp}
\DeclareMathOperator{\PGL}{PGL}

\newcommand{\As}{\operatorname{As}}

\newcommand{\Conj}{\operatorname{Conj}}

\newcommand{\Core}{\operatorname{Core}}

\newcommand{\Inn}{\operatorname{Inn}}

\newcommand{\R}{\operatorname{R}}

\newcommand{\GL}{\operatorname{GL}}
\newcommand{\Z}{\operatorname{Z}}

\newcommand\numberthis{\addtocounter{equation}{1}\tag{\theequation}}

\usepackage[autostyle]{csquotes}

\begin{document}
\title{Constructing an infinite family of quandles from a quandle}
\author{Pedro Lopes}
\author{Manpreet Singh}
\address{Center for Mathematical Analysis, Geometry, and Dynamical Systems, \\
Department of Mathematics, \\
Instituto Superior T\'{e}cnico, Universidade de Lisboa\\
1049-001 Lisbon, Portugal}
\email{pelopes@math.tecnico.ulisboa.pt}
\email{manpreet.math23@gmail.com}


\begin{abstract}
Quandles are self-distributive, right-invertible, idempotent algebras. A group with conjugation for binary operation is an example of a quandle. Given a quandle $(Q, \ast)$ and a positive integer $n$, define $a\ast_n b = (\cdots (a\ast \underbrace{b)\ast \cdots )\ast b}_{n}$, where $a, b \in Q$. Then, $(Q, \ast_n)$ is again a quandle. We set forth the following problem. ``Find $(Q, \ast)$ such that the sequence $\{(Q, \ast_n)\, :\, n\in \mathbb{Z}^+ \}$ is made up of pairwise non-isomorphic quandles.'' In this article we find such a quandle $(Q, \ast)$. We study the general linear group of $2$-by-$2$ matrices over $\mathbb{C}$ as a quandle under conjugation. Its (algebraically) connected components, that is, its conjugacy classes, are subquandles of it. We show the latter are connected as quandles and prove rigidity results about them such as the dihedral quandle of order $3$ is not a subquandle for most of them. Then we consider the quandle which is the projective linear group of $2$-by-$2$ matrices over $\mathbb{C}$  with conjugation, and prove it solves the problem above. In the course of this work we prove a sufficient and necessary condition for a quandle to be latin. This will reduce significantly the complexity of algorithms for ascertaining if a quandle is latin.
\end{abstract}

\maketitle

\section*{Introduction}\label{sec:intro}
Quandles are binary algebras whose axioms are the interpretation of Reidemeister moves of oriented link diagrams. Any group with conjugation for binary operation is a quandle. In the 1980s, Joyce \cite{MR638121,MR2628474} and Matveev \cite{MR672410} proved that the fundamental quandle of an oriented knot is a complete knot invariant up to the orientation of the ambient space and that of the knot. This result extends to orientable non-split links. Since then quandles have been studied extensively and are used to construct new invariants. We refer the readers to \cite{MR3379534,MR3729413,MR3588325} for more details.
\par 

Besides knot theory, the importance of quandles paved into other areas of mathematics, such as Hopf algebras \cite{MR1994219},  solutions to the set-theoretic quantum Yang-Baxter equation \cite{MR1994219,MR1722951, MR2153117}  and Riemannian symmetric spaces \cite{MR217742}. Concurrently, a study of quandles and racks from the algebraic point of view emerged. A (co)homology theory for quandles and racks has been developed in \cite{MR1990571,MR1725613} that led to stronger invariants for knots and knotted surfaces. Algebraic and combinatorial properties of quandles have been investigated, for example, in \cite{MR3718201,MR3948284,MR3981139,MR4075375,MR4330281,MR4290330,MR3582881,MR2253834,MR2253834,MR4377349,MR4014612,MR2194774}. In the last few years new aspects of quandle theory were introduced like covering theory \cite{MR1954330,MR3205568}, extension theory \cite{MR2008876,MR4282648}, commutator theory \cite{MR4203321} and ring theory \cite{MR3915329,MR3977818}.
\par

Given a quandle $(Q, \ast)$ and a positive integer $n$, we define $a\ast_n b = (\cdots (a\ast \underbrace{b)\ast \cdots )\ast b}_{n}$ for $a,b \in Q$. Then, $(Q, \ast_n)$ is again a quandle (see \cite{2022arXiv220412571B}). In this article we present a solution to the following problem. 

\begin{Problem}
Find a quandle $(Q, \ast)$ such that the sequence $\{(Q, \ast_n)\, :\, n\in \mathbb{Z}^+ \}$ is  made up of pairwise non-isomomorphic quandles.
\end{Problem}
\par

In the current article we also address the conjugation quandle on $\GL(2, \mathbb{C})$ and on $\PGL(2, \mathbb{C})$, the latter eventually providing an answer to the problem above (Theorem \ref{the-main-theorem-of multiplying-quandles}). We let $\Conj(\GL(2, \mathbb{C}))$ stand for the conjugation quandle on $\GL(2, \mathbb{C})$. We prove a number of results about $\Conj(\GL(2, \mathbb{C}))$ and about its infinite (algebraically) connected components which are subquandles of it. (We note that since these are topological groups then they are also topological quandles.) Our results are new even from the point of view of conjugacy classes of $\GL(2, \mathbb{C})$. For instance, we prove that the non-trivial connected components of $\Conj(\GL(2, \mathbb{C}))$ are $2$-connected (Theorem \ref{distinct-pm-eigen-values-are-2-connected}, Theorem \ref{opposite-eigen-values-are-2-connected}, Theorem \ref{same-eigen-value-is-2-connected}). This means that some pair of elements in the quandle, say $(a, b)$ requires two elements, say $x, y$, to solve the equation $b=(\cdots ((a\ast z_1)\ast z_2) \cdots )\ast z_n$ and no $(a, b)$ in the quandle requires more than two $x, y$ to solve said equation. In the course of our work on $\Conj(\GL(2, \mathbb{C}))$, we prove a general result concerning $n$-connectedness from which we obtain that a quandle is latin if and only if one of its left multiplications is bijective (Theorem \ref{sufficient-condition-for-Latin-quandle}). This latter result reduces considerably the complexity of the algorithms for checking if a quandle is latin.

Resuming the discussion of our work on $\Conj(\GL(2, \mathbb{C}))$, we also provide a partial classification of the quandles which are the connected components of $\Conj(\GL(2, \mathbb{C}))$ (Theorem \ref{thm:6.1}, Theorem \ref{thm:6.2}, Theorem \ref{thm:6.3}). In particular, we prove that there are infinitely many non-isomporphic such quandles. Furthermore, we prove results about rigidity of these quandles. For instance, we prove that $\R_3$ cannot be a subquandle of most of them and also that certain strings of two elements do not fix elements by right multiplication (Theorem \ref{prop:noR3}, Theorem \ref{thm:6.11}).

\par 

\par 

The present article is organized as follows. In Section \ref{sec:preliminaries}, we recall the basics of quandles that are used later. In Section \ref{sec:sufficient-conditions}, we find a sufficient and necessary condition for a quandle to be Latin (respect., $n$-connected). In Section \ref{sec:properties-of-composition}, we note down some results of generating quandles by multiplying a quandle operation with itself. In Section \ref{sec:connected-components-of-GL(2,C)}, we prove that all the non-trivial connected components of $\Conj(\GL(2, \mathbb{C}))$ are $2$-connected. In Section \ref{sec:evidence-for-infinitely-number-of-quandles}, we prove that, for $n \geq 1$, the $n$-conjugation quandles over $\PGL(2, \mathbb{C})$ are pairwise distinct. In Section \ref{classification-of-connected-components}, we provide a partial classification of the connected components of $\Conj(\GL(2, \mathbb{C})$. In Section \ref{sec:future_work}, we provide some ideas for future work.

\section{Preliminaries}\label{sec:preliminaries}

In this section, we recall the definition of quandles and the associated terms with some examples.

\begin{definition}\label{Q}
A {\it quandle} is a non-empty set $Q$ together with a binary operation $*$ satisfying the following axioms:
\begin{enumerate}[label={\textbf{Q}\arabic*}]
\item  $x*x=x$\, for all $x\in Q$.
\item  For each $x, y \in Q$, there exists a unique $z \in Q$ such that $x=z*y$.
\item  $(x*y)*z=(x*z)*(y*z)$\, for all $x,y,z\in Q$.
\end{enumerate}
\end{definition}

The axiom {\textbf Q2} is equivalent to the bijectivity of the right multiplication by each element of $Q$. This gives a dual binary operation $*^{-1}$ on $Q$ defined as $x*^{-1}y=z$ if $x=z*y$. Thus, the axiom {\textbf Q2} is equivalent to saying that
\begin{equation*}
(x*y)*^{-1}y=x\qquad\textrm{and}\qquad\left(x*^{-1}y\right)*y=x
\end{equation*}
for all $x,y\in Q$, and hence it allows us cancellation from the right. The axioms {\textbf Q1} and {\textbf Q3} are referred to as idempotency and distributivity axioms, respectively. The idempotency and cancellation law gives $x*^{-1}x=x$ for all $x\in Q$.
\par
\begin{examples}
The following are some examples of quandles.
\begin{itemize}
\item Let $S$ be any non-empty set. Then define a binary operation $*$ on $S$ as $x*y=x$ for all $x, y \in S$. Then $(S, *)$ becomes a quandle and is called {\it trivial quandle}.
\item If $G$ is a group and $n \in \mathbb{Z}$, then the binary operation $x*y=y^{-n}xy^n$ turns $G$ into the quandle $\Conj_n(G)$ called the $n$-\textit{conjugation quandle} of $G$. For $n=1$, the quandle is simply denoted by $\Conj(G)$.
\item A group $G$ with the binary operation $x*y=yx^{-1}y$ turns $G$ into the quandle $\Core(G)$ called the \textit{core quandle} of $G$. In particular, if $G$ is a cyclic group of order $n$, then it is called the \textit{dihedral quandle} and is denoted by $\R_n$. Usually, one writes $\R_n=\{0, 1, \ldots, n-1\}$ with $i*j=2j-i \mod n$.
\end{itemize}
\end{examples}
In Section \ref{sec:properties-of-composition} we elaborate about $\Conj_n$ and $\Core$ as functors from the category of groups $\pmb \Grp$ to the category of quandles $\pmb \Qnd$.
\par 

A \textit{homomorphism of quandles} $P$ and $Q$ is a map $\phi:P\to Q$ with $\phi(x*y)=\phi(x)*\phi(y)$ for all $x,y\in P$. By the cancellation law in $P$ and $Q$, we get $\phi(x*^{-1}y)=\phi(x)*^{-1}\phi(y)$ for all $x,y\in P$. The quandle axioms are equivalent to saying that for each $y\in Q$, the map $R_y:Q\to Q$ given by $R_y(x)=x* y$ is an automorphism of $Q$ fixing $x$. The group generated by such automorphisms is called the group of {\it inner automorphisms} of $Q$, denoted $\Inn(Q)$. 
\medskip

We recall some relevant definitions. A quandle $Q$ is said to be 
\begin{itemize}
\item {\it connected} if the group $\Inn(Q)$ acts transitively on $Q$. For example, the dihedral quandle $\R_{2n+1}$ is connected whereas $\R_{2n}$ is not. 
\item {\it involutory} if $x*^{-1}y=x* y$ for all $x,y\in Q$. For example, the core quandle $\Core(G)$ of any group $G$ is involutory.
\item {\it latin} if the left multiplication $L_x:Q\to Q$ defined by $L_x(y)=x* y$ is a bijection for each $x\in Q$. 
\item {\it $1$-connected} if $L_x$ is a surjective map for each $x \in Q$. Each latin quandle is obviously $1$-connected, but the converse is not true. For example, $\Core(\mathbb{C}^*)$ is $1$-connected but not latin.
\item {\it $n$-connected} if it is not $k$-connected for $ 1 \leq k \leq n-1$ and for any $x, y \in Q$, we have $x=y * z_1 * \cdots * z_l$ and $y=x * z_1' * \cdots * z_{l'}'$ where $1 \leq l, l' \leq n$ and  $z_i, z_j' \in Q$ for all $1\leq i \leq l$ and $1\leq j \leq l'$. Note that this property is a quandle invariant.
\end{itemize}

\medskip

The \textit{associated group} $\As(Q)$ of a quandle $Q$ is the group with the set of generators as $\{e_x~|~x \in Q\}$ and the defining relations as
\begin{equation*}
e_{x*y}=e_{y}^{-1} e_{x} e_{y}
\end{equation*}
for all $x,y\in Q$. For example, if $Q$ is a trivial quandle, then $\As(Q)$ is the free abelian group of rank equal to the cardinality of $Q$. The {\it natural map}
\begin{equation*}
\eta: Q \to \Conj(\As(Q))
\end{equation*}
given by $\eta(x)=e_{x}$ is a quandle homomorphism. The map $\eta$ is not injective in general. 
\par 

If $Q$ is a quandle, then by \cite[Lemma 4.4.7]{MR2634013}, we can write
\begin{equation*}
x*^d\left(y*^e z\right)=\left(\left(x*^{-e} z\right)*^dy\right)*^ez\quad\textrm{(called the {\it left association identity})}
\end{equation*}
for all $x,y,z\in Q$ and $d,e\in\{-1,1\}$. Henceforth, we will write a left-associated product
\begin{equation*}
\left(\left(\cdots\left(\left(a_0*^{e_1}a_1\right)*^{e_2}a_2\right)*^{e_3}\cdots\right)*^{e_{n-1}}a_{n-1}\right)*^{e_n}a_n
\end{equation*}
simply as 
\begin{equation*}
a_0*^{e_1}a_1*^{e_2}\cdots*^{e_n}a_n.
\end{equation*}

A repeated use of left association identity gives the following result \cite[Lemma 4.4.8]{MR2634013}.

\begin{lemma}\label{lem4}
The product 
\begin{equation*}
\left(a_0*^{d_1}a_1*^{d_2}\cdots*^{d_m}a_m\right)*^{e_0}\left(b_0*^{e_1}b_1*^{e_2}\cdots*^{e_n}b_n\right)
\end{equation*}
\noindent of two left-associated forms $a_0*^{d_1}a_1*^{d_2}\cdots*^{d_m}a_m$ and $b_0*^{e_1}b_1*^{e_2}\cdots*^{e_n}b_n$ in a quandle can again be written in a left-associated form as
\begin{equation*}
a_0*^{d_1}a_1*^{d_2}\cdots*^{d_m}a_m*^{-e_n}b_n*^{-e_{n-1}}b_{n-1}*^{-e_{n-2}}\cdots*^{-e_1}b_1*^{e_0}b_0*^{e_1}b_1*^{e_2}\cdots*^{e_n}b_n.
\end{equation*}
\end{lemma}

Thus, any product of elements of a quandle $Q$ can be expressed in the canonical left-associated form $a_0*^{e_1}a_1*^{e_2}\cdots*^{e_n}a_n$, where $a_0\neq a_1$, and for $i=1,2,\ldots,n-1$, $e_i=e_{i+1}$ whenever $a_i=a_{i+1}$.

\par 

Let $Q$ be a quandle, and $x, y \in Q$. Then for $n \in \mathbb{Z}^+$, we let $x *^{\epsilon n} y=x\underbrace{*^\epsilon y*^ \epsilon y* ^ \epsilon \cdots*^ \epsilon y}_{n \textrm{ times }}$, where $\epsilon= \pm 1$. We say $Q$ is {\it $n$-quandle} if $x*^n y=x$ for all $x, y \in Q$. For example, core quandles are $2$-quandles.

\begin{remark}
Note that if $Q$ is $n$-quandle, then it is also $nm$-quandle, where $m \in \mathbb{Z}^+$.
\end{remark}
\par 

\begin{definition}\label{def:finite-type}
For $n \in \mathbb{Z}^+$, a quandle $Q$ is said to be of {\it type} $n$, if $n$ is the least number satisfying $p *^n q=p$ for all $p,q \in Q$. We denote the type of $Q$ by $\type(Q)$. If there does not exist any such $n$, then we say the quandle is of type $\infty$.
\end{definition}
If for a quandle $Q$, $\type(Q)$ is finite, then we say $Q$ is of {\it finite type}. Clearly finite quandles and $n$-quandles are of finite type.

\par 

A subset $P$ of a quandle $Q$ is called {\it subquandle} if it is a quandle under the same binary operation of $Q$. The terms  defined for quandles in this section are similarly defined for subquandles, for example, if $Q$ is a quandle and $P$ its subquandle, then we say $P$ is an $n$-subquandle, if for all $x, y \in P$ we have $x*^ny=x$.
\medskip

\section{A sufficient and necessary condition for a quandle to be latin (resp., $n$-connected)}\label{sec:sufficient-conditions}

In this section, we find a sufficient and necessary condition for a quandle to be latin and $n$-connected.
\begin{theorem}\label{sufficient-condition-for-Latin-quandle}
Let $Q$ be a quandle. Then $Q$ is latin if and only if there exists $x \in Q$ such that $L_x: Q \to Q$ is bijective.
\end{theorem}
\begin{proof}By definition of latin quandle the necessity is obvious.
\par 
Now we will prove the converse.
Let $ y \in Q$ and $y \neq x$. Since $L_x$ is surjective, there exists $\alpha \in Q$ such that $L_x(\alpha)=y$, which implies $x=R_{\alpha}^{-1} (y)$.
\begin{claim}
$ L_y $ is surjective. 
\end{claim}
Let $a \in Q$.  We need to find $b \in Q$ such that $L_y(b)=a$. Consider $a'=R_{\alpha}^{-1}(a)$. Again the surjectivity of $L_x$ implies the existence of $\beta \in Q$ such that 

\begin{alignat*}{3}
&~&&L_x(\beta)&&=a'\\
&\iff &&R_{\alpha}^{-1} (y) * \beta &&=R_{\alpha}^{-1}(a)\\
&\iff &&R_{\alpha}(R_{\alpha}^{-1}(y) * \beta)&&=a\\
&\iff &&y*R_{\alpha}(\beta)&&=a.
\end{alignat*}
Taking $b=R_{\alpha}(\beta)$ implies that $L_y(b)=a$, and thus proves the claim.
\begin{claim}
$L_y$ is injective.
\end{claim}
Let $b_1, b_2\in Q$. If $L_y(b_1)=a=L_y(b_2)$, then $R_{\alpha}^{-1}(y * b_1)=R_{\alpha}^{-1}(a)=R_{\alpha}^{-1}(y* b_2)$ which is equivalent to $L_x(R_{\alpha}^{-1}(b_1))=R_{\alpha}^{-1}(a)=L_x(R_{\alpha}^{-1}(b_2))$. But since $L_x$ is injective, we have $R_{\alpha}^{-1}(b_1)=R_{\alpha}^{-1}(b_2)$, which further implies that $b_1=b_2$. Thus $L_y$ is injective.
\end{proof}

In the HAP package \cite{HAP1.47}, finite Latin quandles are recognized by testing whether all left translations are permutations or not. As an application of Theorem \ref{sufficient-condition-for-Latin-quandle}, the time complexity of detecting Latin quandles can be reduced.

Using \cite[Proposition 3.3]{MR4377349} and Theorem \ref{sufficient-condition-for-Latin-quandle}, we have the following result.

\begin{theorem}
Let $Q$ be a finite quandle and $x \in Q$. Suppose that the cycles of $R_x$ have pairwise distinct lengths and for every $y \in Q$, $R_y$ has a unique fixed point. Then $Q$ is Latin.
\end{theorem}

In the following theorem, we find a sufficient and necessary condition for a quandle to be $n$-connected.

\begin{theorem}\label{sufficinet-condition-for-n-connected}
Let $Q$ be a quandle and $x \in Q$ such that for any element $a \in Q$ there exist $x_i\textrm{'s} \in Q$, where $1\leq i \leq n$ such that $R_{x_{i}}\cdots R_{x_1}(x)=a$. If $Q$ is not $m$-connected, where $1 \leq m \leq n-1$, then $Q$ is $n$-connected.
\end{theorem}
\begin{proof}
Let $y,z \in Q$ and $y \neq x$. By assumption, there exist $x_1, \ldots, x_j$, where $j\leq n$, such that $y=R_{x_j}R_{x_{j-1}}\cdots R_{x_1}(x)$. Again by assumption, there exist $x_1', \ldots, x_k'$, where $k\leq n$, such that 
\begin{alignat*}{3}
&~&&R_{x_k'}\cdots R_{x_1'}(x)&&=R_{x_1}^{-1} \cdots R_{x_j}^{-1}(z)\\
&\iff &&R_{x_k'}\cdots R_{x_1'} R_{x_1}^{-1} \cdots R_{x_j}^{-1}(y)&&= R_{x_1}^{-1} \cdots R_{x_j}^{-1}(z)\\
&\iff &&R_{y_k}\cdots R_{y_1}(y)&&=z
\end{alignat*}
where $y_l=R_{x_j}\cdots R_{x_1}(x_l')$ for $1 \leq l \leq k \leq n$. This completes the proof.
\end{proof}
\medskip

\section{Properties of multiplication of quandle operation}\label{sec:properties-of-composition}
In this section we study properties of quandles which are generated by iterating its quandle operation a given  number of times (see Proposition \ref{generating-quandles-by-multiplying-quandle-operation})-multiplication of quandle operation (see \cite{2022arXiv220412571B}).
\par 

The following result is proved in \cite[Theorem 4.11]{2022arXiv220412571B} in a general setting, but we are noting it down for completeness.
\begin{proposition}\label{generating-quandles-by-multiplying-quandle-operation}
Let $(Q, *)$ be a quandle. Then for each $n \in \mathbb{Z}^+$, $(Q, *_n)$ is a quandle, where $a*_n b := a *^n b$ for $a, b \in (Q, *_n).$
\end{proposition}
\begin{proof}
We just need to verify that $(Q, *_n)$ satisfies the quandle axioms. Clearly it satisfies \textbf{Q1} axiom. Let $x, y \in (Q,*_n)$ and $z=x*^{-n}y$ in $Q$. Then $x=z*_ny$, and $z$ is unique because $R_y^{-n}$ is a bijection. Thus \textbf{Q2} is also satisfied. Now take $x,y, z \in (Q, *_n)$. We need to show that $(x*_ny)*_nz=(x*_nz)*_n(y*_nz)$. Using \textbf{Q3} recursively $n$ times, we note that $(x*_ny)*z=(x*z)*_n(y*z)$. Thus $(x*_ny)*_nz=(x*_nz)*_n(y*_nz)$.
\end{proof}

Now onwards, for a quandle $(Q, *)$, we denote the quandle $(Q, *_n)$ by $\mathcal{Q}_n(Q)$.

\begin{remark}
If $Q$ is a quandle and $\type(Q)$ is $\infty$, then  $\type\big(\mathcal{Q}_n(Q)\big)$ is $\infty$ for all $n \in \mathbb{Z}^+$. 
\end{remark}
\par 

If $Q$ is a quandle such that Type$(Q)$ is finite, say $n$, then $\mathcal{Q}_n(Q)$ is a trivial quandle. Furthermore, $\mathcal{Q}_m(Q) \cong \mathcal{Q}_{(m+n)}(Q)$. Thus the following questions are in order.

\begin{question}
Let $Q$ be a quandle such that $\type(Q)$ is $\infty$. Does there exist $n \in \mathbb{Z}^+$, where $n \neq 1$, such that $Q \cong \mathcal{Q}_n(Q)$?
\end{question}

\begin{question}
Let $Q$ be a quandle such that $\type(Q)$ is $\infty$. Does there exist $m, n >1$, where $m \neq n$, such that $\mathcal{Q}_m(Q) \cong \mathcal{Q}_n(Q)$?
\end{question}
In Section \ref{sec:evidence-for-infinitely-number-of-quandles},  we will answer the above questions for $\Conj(\GL(2,\mathbb{C}))$, and $\Conj(\PGL(2, \mathbb{C}))$.
\medskip
\par 

We note that if $f: (Q, *) \to (Q', *')$ is a quandle homomorphism, then the map $\mathcal{Q}_n(f): \mathcal{Q}_n(Q) \to \mathcal{Q}_n(Q')$ defined as $\mathcal{Q}_n(f)(x*_n y)=f(x )*'_n f(y)$ is a quandle homomorphism. Thus we have the following result.
\begin{proposition}\label{functor}
Let $\pmb \Qnd$ denote the category of quandles. Then we have a functor $\boldsymbol{\mathcal{Q}_n}: {\pmb \Qnd} \to {\pmb \Qnd}$ defined as $\boldsymbol{\mathcal{Q}_n}(Q)=\mathcal{Q}_n(Q)$ and $\boldsymbol{\mathcal{Q}_n}(f: Q_1 \to Q_2)$ as $\mathcal{Q}_n(f):\mathcal{Q}_n(Q_1) \to \mathcal{Q}_n(Q_2)$.
\end{proposition}

\begin{remark}
Using Proposition \ref{functor}, we note that the functor $\Conj_n: \pmb \Grp \to \pmb \Qnd$ is the composite of $\Conj: \pmb \Grp \to \pmb \Qnd$ and $\boldsymbol{\mathcal{Q}_n}: {\pmb \Qnd} \to {\pmb \Qnd}$ functors, that is $\Conj_n=\boldsymbol{\mathcal{Q}_n} \circ \Conj$.
\end{remark}

For each quandle $Q$, there is a quandle homomorphism $\sigma: Q \to \Conj(\Inn(Q))$  defined as $\sigma(q)=R_q$.

\begin{corollary}
Let $Q$ be a quandle. Then for each $n$, the following are true:
\begin{enumerate}
\item The map $\mathcal{Q}_n(\eta): \mathcal{Q}_n(Q) \to \Conj_n(\As(Q))$ defined as $x \mapsto e_x$ is a quandle homomorphism
\item the map $\mathcal{Q}_n(\sigma): \mathcal{Q}_n(Q) \to \Conj_n(\Inn(Q))$ defined as $x \mapsto R_x$ is a quandle homomorphism.
\end{enumerate}
\end{corollary}
\begin{proof}
The result follows from Proposition \ref{functor}.
\end{proof}

\begin{proposition}
Let $Q$ be a quandle. Then for each $n$, the following holds:
\begin{enumerate}
\item the map $\psi: \As(\mathcal{Q}_n(Q))\to \As(Q)$ defined on the generators as $\psi(e_q)=e_q^n$ is a group homomorphism.
\item the map $\eta_n: \mathcal{Q}_n(Q) \to \Conj(\As(Q))$ defined as $x \mapsto e_x^n$ is a quandle homomorphism.
\item the map $\sigma_n: \mathcal{Q}_n(Q) \to \Conj(\Inn(Q))$ defined as $x \mapsto R_x^n$ is a quandle homomorphism.
\end{enumerate} 
\end{proposition}
\begin{proof}
Let $p,q \in \mathcal{Q}_n(Q)$. Then the map $\psi: \As(\mathcal{Q}_n(Q))\to \As(Q)$ defined on the generators as $\psi(e_q)=e_q^n$ is a group homomorphism if and only if
\begin{alignat*}{2}
&\psi(e_{p*_n q})&&=\psi(e_q^{-1} e_p e_q)\\
\iff &e_{p*^n q}^n&&=e_q^{-n} e_p^n e_q^n\\
\iff & (e_q^{-n} e_p e_q^n)^n&&= e_q^{-n} e_p^n e_q^n,
\end{alignat*}
which is true.
Thus $\psi: \As(\mathcal{Q}_n(Q)) \to \As(Q)$ is a group homomorphism.
The second assertion follows from the fact that the natural map $\eta:\mathcal{Q}_n(Q) \to \Conj(\As(\mathcal{Q}_n(Q)))$ is a quandle homomorphism. The last assertion follows from the fact that there is a natural surjective group homomorphism $\As(Q) \to \Inn(Q)$ given  by $e_x \mapsto R_x$.
\end{proof}

%%%%%%%%%%%%%%%%%%%%%%%%%%%%%%%%%%%%%%%%%%%%%%%%%%%%%%%%%%%%%%%%%
%%%%%%%%%%%%%%%%%%%%%%%%%%%%%%%%%%%%%%%%%%%%%%%%%%%%%%%%%%%%%%%%%
%%%%%%%%%%%%%%%%%%%%%%%%%%%%%%%%%%%%%%%%%%%%%%%%%%%%%%%%%%%%%%%%%

\section{Connected components of  $\Conj(\GL(2, \mathbb{C}))$}\label{sec:connected-components-of-GL(2,C)}

Consider the group $\GL(2,\mathbb{C})$ of invertible $2$-by-$2$ matrices over the field of complex numbers $\mathbb{C}$. The  {\it projective general linear group} $\PGL(2, \mathbb{C})$ of degree $2$ over $\mathbb{C}$ is the quotient of $\GL(2, \mathbb{C})$ by its center $\Z(\GL(2, \mathbb{C}))\cong \mathbb{C}^*$. It is known that each connected component of a quandle is a subquandle. In this section we will prove that each non-trivial connected components of $\Conj(\GL(2, \mathbb{C}))$ are $2$-connected.

It is well known that there are exactly three families of conjugacy classes in $\GL(2,\mathbb{C})$, namely:
\begin{enumerate}
\setlength\itemsep{1em}
\item for each $\lambda_1 \neq \lambda_2 \in \mathbb{C}^*$, the conjugacy class of 
$\begin{pmatrix}
  \lambda_1 & 0 \\
  0 & \lambda_2
\end{pmatrix}$,
\item for each $\lambda \in \mathbb{C}^*$, the conjugacy class of 
$\begin{pmatrix}
  \lambda & 1 \\
  0 & \lambda
\end{pmatrix}$,
\item for each $\lambda \in \mathbb{C}^*$, the conjugacy class of 
$\begin{pmatrix}
  \lambda &0 \\
  0 & \lambda
\end{pmatrix}$, which is a singleton set.
\end{enumerate}

\begin{remark}
We disregard $(3)$ above because each of its quandles corresponds to the trivial quandle with one element.
\end{remark}

\begin{notation}
\begin{enumerate}
\item For $\lambda, \lambda_1, \lambda_2 \in \mathbb{C}^*$, we denote the set of conjugacy class of $\begin{pmatrix}
\lambda_1 & 0\\
0 & \lambda_2
\end{pmatrix}$ and $\begin{pmatrix}
\lambda & 1\\
0 & \lambda
\end{pmatrix}$ by $M_{\lambda_1, \lambda_2}$ and $M_{\lambda}$, respectively. Note that $M_{\lambda_1, \lambda_2}=M_{\lambda_2, \lambda_1}$
\item We denote the matrix $\begin{pmatrix}
\lambda_1 & 0\\ 0 & \lambda_2
\end{pmatrix}$ by $D(\lambda_1, \lambda_2)$.
\item For any $X \in \GL(2, \mathbb{C})$, the trace of $X$ and determinant of $X$ are denoted by $\tr X$ and $\det X$, respectively.
\end{enumerate}  
\end{notation}

\begin{lemma}\label{non-trivial-subquandles} 
Each connected component of the quandle $\Conj(\GL(2,\mathbb{C}))$ having more than one element is a non-trivial (infinite) subquandle.
\end{lemma}
\begin{proof}
Let $\lambda_1, \lambda_2, \alpha \in \mathbb{C}^*$ and $\lambda_1 \neq \lambda_2$. Then 
\begin{align*}
\begin{pmatrix}
  \lambda_1 &  0\\
  0 & \lambda_2
\end{pmatrix}^{-1}
\begin{pmatrix}
  \lambda_1 & \alpha \\
  0 & \lambda_2
\end{pmatrix}
\begin{pmatrix}
  \lambda_1 & 0 \\
  0 & \lambda_2
\end{pmatrix}&=\frac{1}{\lambda_1 \lambda_2}
\begin{pmatrix}
  \lambda_2 & 0 \\
  0 & \lambda_1
\end{pmatrix}
\begin{pmatrix}
  \lambda_1 & \alpha \\
  0 & \lambda_2
\end{pmatrix}
\begin{pmatrix}
  \lambda_1 & 0 \\
  0 & \lambda_2
\end{pmatrix}\\
&=\begin{pmatrix}
\lambda_1 & \alpha \lambda_2/\lambda_1\\
0 & \lambda_2
\end{pmatrix}\\
&\neq \begin{pmatrix}
\lambda_1 & \alpha\\
0 & \lambda_2
\end{pmatrix}.
\end{align*}
Thus $M_{\lambda_1, \lambda_2}$ is a non-trivial subquandle; it is also an infinite quandle since $\alpha \in \mathbb{C}^*$.
\par 

Let $\lambda, \alpha \in \mathbb{C}^*$. Then

\begin{align*}
\begin{pmatrix} \lambda & \alpha \\ 0 & \lambda \end{pmatrix} 
*
\begin{pmatrix} \lambda & 0\\ \alpha & \lambda \end{pmatrix} 
&=\frac{1}{\lambda^2}
\begin{pmatrix} \lambda & 0 \\ -\alpha & \lambda \end{pmatrix}
\begin{pmatrix} \lambda & \alpha \\ 0 & \lambda \end{pmatrix}
\begin{pmatrix} \lambda & 0\\ \alpha & \lambda \end{pmatrix}\\
&=\frac{1}{\lambda^2}\begin{pmatrix}
\lambda^3 + \alpha^2 \lambda & \alpha \lambda^2 \\ -\alpha^3 & -\alpha^2 \lambda+ \lambda^3 \end{pmatrix}\\
&\ne
\begin{pmatrix} \lambda & 1 \\ 0 & \lambda \end{pmatrix}
\end{align*}

\iffalse{
 and $P=\begin{pmatrix}
\lambda & 0\\
1 & \lambda
\end{pmatrix}$. By induction one can easily prove that $P^n=\begin{pmatrix}
\lambda^n & 0\\
n\lambda^{n-1} & \lambda^n
\end{pmatrix}$ for all $n \in \mathbb{Z}^+$. Now,
\begin{align*}
\begin{pmatrix}
\lambda & 1\\
0 & \lambda
\end{pmatrix} \underbrace{* P * \cdots * P}_{n ~\textrm{times}} &=
P^{-n} \begin{pmatrix}
\lambda & 1\\
0 & \lambda
\end{pmatrix} P^{n}\\
&= \begin{pmatrix}
\lambda+ n/\lambda & 1\\
-(n/\lambda)^2 & \lambda -n/\lambda
\end{pmatrix}\\
&\neq \begin{pmatrix}
\lambda & 1 \\
0 & \lambda
\end{pmatrix}.
\end{align*}}\fi

Thus $M_{\lambda}$ is a non-trivial subquandle; it is also an infinite quandle since $\alpha \in \mathbb{C}^*$.
\end{proof}

\medskip

\subsection{The proof that $M_{\lambda_1, \lambda_2}$ is $2$-connected, where $\lambda_1\neq \pm \lambda_2$}
In this subsection we will prove that for $\lambda_1, \lambda_2 \in \mathbb{C}^*$ where $\lambda_1 \neq \pm \lambda_2$, the quandle $M_{\lambda_1, \lambda_2}$ is $2$-connected (see Theorem \ref{distinct-pm-eigen-values-are-2-connected}).
\par

\begin{lemma}\label{joining D(1, -1) to D(-1,1)}
Let $\lambda_1, \lambda_2 \in \mathbb{C}^*$ and $\lambda_1 \neq \lambda_2$. Then there exists $X \in \M_{\lambda_1, \lambda_2}$ such that $D(\lambda_2, \lambda_1) *X=D(\lambda_1, \lambda_2)$ if and only if $\lambda_1=-\lambda_2$.
\end{lemma}
\begin{proof}
For 
$X=\begin{pmatrix}
  x_{1,1} & x_{1,2} \\
  x_{2,1} & x_{2,2}
\end{pmatrix}$ in $\GL(2, \mathbb{C})$, we have: 
\begin{align}
&\begin{pmatrix}
  x_{1,1} & x_{1,2} \\
  x_{2,1} & x_{2,2}
\end{pmatrix}\begin{pmatrix}
  \lambda_1 & 0 \\
  0 & \lambda_2
\end{pmatrix} = \begin{pmatrix}
  \lambda_2 & 0 \\
  0 & \lambda_1
\end{pmatrix}\begin{pmatrix}
  x_{1,1} & x_{1,2} \\
  x_{2,1} & x_{2,2}
\end{pmatrix} \label{eq1}\\
& \Longleftrightarrow~ \begin{cases}
x_{1,1}(\lambda_1-\lambda_2)=0\\
x_{2,2}(\lambda_1-\lambda_2)=0
\end{cases} \notag \\
&\Longleftrightarrow~ \begin{cases}
x_{1,1}=0\\
x_{2,2}=0
\end{cases} (~\textrm{since}~ \lambda_1\neq \lambda_2). \label{eq2}
\end{align}
Thus for $X$ to be in $\M_{\lambda_1, \lambda_2}$, the trace of $X$ must be zero. This proves the necessity of $\lambda_1=-\lambda_2.$

\par 
Now let $\lambda_1=-\lambda_2=\lambda$, and take $X=\begin{pmatrix}
0 & 1\\
\lambda^2 & 0.
\end{pmatrix}$
Then clearly $X$ is in $\M_{\lambda_1, \lambda_2}$ and satisfies conditions in \eqref{eq2} which further implies that it satisfies $\eqref{eq1}$. Thus $D(\lambda_2, \lambda_1) *X=D(\lambda_1, \lambda_2)$.
\end{proof}

\begin{lemma}\label{A}
Let $\lambda_1, \lambda_2 \in \mathbb{C}^*$ such that $\lambda_1\neq \pm \lambda_2$, and $A=\begin{pmatrix}
a & b\\ c & d
\end{pmatrix} \in M_{\lambda_1, \lambda_2}$ such that either $b\ne 0$ or $c \ne 0$. Then the following holds:
\begin{enumerate}[label={(\roman*)}]
\item there exists $X, Y\in M_{\lambda_1, \lambda_2}$ such that $D(\lambda_1, \lambda_2)= A*X$ and $D(\lambda_2, \lambda_1)=A*Y.$ \label{A1}\\
\item there exists $X, Y\in M_{1/\lambda_1, 1/\lambda_2}$ such that $D(\lambda_1, \lambda_2)* X^{-1}= A$ and $D(\lambda_2, \lambda_1)* Y^{-1}=A.$ \label{A2}
\end{enumerate}
\end{lemma}
\begin{proof}We will prove part \ref{A1}. The proof of part \ref{A2} is done along the same lines.
\par 

We will show that there exists $X \in M_{\lambda_1, \lambda_2}$ such that $D(\lambda_1, \lambda_2)=A * X =X^{-1} A X$, that is, we will prove that $X$ is a diagonalizable matrix in $M_{\lambda_1, \lambda_2}$.
We will prove the existence of $X$ when $c \neq 0$, the other case where $b \ne 0$ is similar.
The system of equations for the eigenvectors of $A$ is ($i\in \{1, 2\}$) $$\begin{pmatrix}
0  \\
0
\end{pmatrix}=\begin{pmatrix}
a-\lambda_i & b \\
c & d-\lambda_i
\end{pmatrix}\begin{pmatrix}
x  \\
y
\end{pmatrix}$$ so that
\begin{align*}
\begin{pmatrix}
a-\lambda_i & b &\bigm| & 0\\
c  & d-\lambda_i & \bigm| & 0
\end{pmatrix}&& \longrightarrow &&&\begin{pmatrix}
a-\lambda_i & b &\bigm| & 0\\
1  & (d-\lambda_i)/c & \bigm| & 0
\end{pmatrix}\\[1em]
\longrightarrow \begin{pmatrix}
0 & b- (a-\lambda_i)(d-\lambda_i)/c&\bigm| & 0\\
1  & (d-\lambda_i)/c & \bigm| & 0
\end{pmatrix}&& \longrightarrow &&&\begin{pmatrix}
1  & (d-\lambda_i)/c & \bigm| & 0\\
0 & 0&\bigm| & 0
\end{pmatrix},\\
\end{align*}
where in the last passage we used the characteristic equation.
The eigenvectors are then:
$$\begin{pmatrix}
u(\lambda_1-d)/c  \\
u
\end{pmatrix} \qquad \text{ and } \qquad \begin{pmatrix}
v(\lambda_2-d)/c  \\
v
\end{pmatrix}
$$
where we keep the scalars $u ~(\neq 0)$ and $v~ (\neq 0)$ for later to choose them appropriately so that the diagonalizing matrix, $$X=\begin{pmatrix}
u(\lambda_1-d)/c & v(\lambda_2-d)/c \\
u                & v
\end{pmatrix} ,$$ is in $\M_{\lambda_1, \lambda_2}$.
Thus, we need
\begin{align}
&\begin{cases}
\lambda_1 + \lambda_2 = \tr X = v + u(\lambda_1-d)/c\\
\lambda_1\lambda_2 = \det X= uv(\lambda_1-\lambda_2)/c
\end{cases} \label{eq9}\\
\iff&\begin{cases}
v= \lambda_1 + \lambda_2 - u(\lambda_1-d)/c\\
\lambda_1\lambda_2 = u\big[ \lambda_1 + \lambda_2 - u(\lambda_1-d)/c \big](\lambda_1-\lambda_2)/c
\end{cases} \notag\\
\Longleftrightarrow&\begin{cases}
0=u^2(\lambda_1-d)(\lambda_1-\lambda_2)/c^2 - u(\lambda_1 + \lambda_2)(\lambda_1-\lambda_2)/c + \lambda_1\lambda_2 \\
v= \lambda_1 + \lambda_2 - u(\lambda_1-d)/c
\end{cases} \label{eq10}
\end{align}
Since $\lambda_1 \lambda_2 \neq 0$, the quadratic equation in $u$ cannot have solutions $u=0$ and $u= c(\lambda_1+ \lambda_2)/(\lambda_1-d)$ (note that $\lambda_1 \neq \lambda_2$), which further implies that $v =0$ is not a solution of the above equations. Thus, choosing a solution of the quadratic equation for $u$ and obtaining $v$ in terms of this $u$ provides $X=X(u,v)\in \M_{\lambda_1, \lambda_2}$, such that $A * X =D(\lambda_1, \lambda_2)$.
\par 
Similarly, one can solve $u'$ and $v'$ for a diagonalizing matrix $$Y=\begin{pmatrix}
u'(\lambda_2-d)/c & v'(\lambda_1-d)/c \\
u'                & v'
\end{pmatrix}$$ to be in $\M_{\lambda_1, \lambda_2}$. Thus there exists $Y \in M_{\lambda_1, \lambda_2}$ such that $A*Y=D(\lambda_2, \lambda_1)$.
\end{proof}

\begin{theorem}\label{distinct-pm-eigen-values-are-2-connected}
Let $\lambda_1, \lambda_2 \in \mathbb{C}^*$, and $\lambda_1 \neq -\lambda_2$. Then $M_{\lambda_1, \lambda_2}$ is a connected quandle. If $\lambda_1\neq \pm \lambda_2$, then $M_{\lambda_1, \lambda_2}$ is a $2$-connected quandle.
\end{theorem}
\begin{proof}
If $\lambda_1=\lambda_2=\lambda$, then being a trivial quandle with one element, $M_{\lambda, \lambda}$ is connected.
\par 

Now suppose $\lambda_1\neq \pm \lambda_2$. By {\textbf Q1}-axiom $D(\lambda_1, \lambda_2)=D(\lambda_1, \lambda_2) * D(\lambda_1, \lambda_2)$. By Lemma \ref{joining D(1, -1) to D(-1,1)}, there does not exist any $Y \in M_{\lambda_1, \lambda_2}$ such that $D(\lambda_1, \lambda_2)=D(\lambda_2, \lambda_1)*Y$. Thus $M_{\lambda_1, \lambda_2}$ cannot be $1$-connected. Furthermore, by Lemma \ref{A}, for any $A=\begin{pmatrix}
a & b\\
c & d
\end{pmatrix} \in M_{\lambda_1, \lambda_2}$ such that either $c\neq 0$ or $b \neq 0$, there exist $X_1, Y_1, X_2, Y_2 \in M_{\lambda_1, \lambda_2}$ such that $A*X_1=D(\lambda_1, \lambda_2)$, $A*Y_1=D(\lambda_2, \lambda_1)$, $A=D(\lambda_1, \lambda_2)*X_2$, and $A=D(\lambda_2, \lambda_1)*Y_2.$ Thus for any $A \in M_{\lambda_1, \lambda_2}$, either there exists $X\in M_{\lambda_1, \lambda_2}$ such that $D(\lambda_1, \lambda_2)*X=A$ or there exist $Y, Z \in M_{\lambda_1, \lambda_2}$ such that $D(\lambda_1, \lambda_2) *Y*Z=A$.
Thus by Theorem \ref{sufficinet-condition-for-n-connected}, $M_{\lambda_1, \lambda_2}$ is $2$-connected.
\end{proof}

\medskip

\subsection{The proof that $M_{\lambda, -\lambda}$ is $2$-connected}
In this subsection, we will prove that for $\lambda \in \mathbb{C}^*$, the quandle $M_{\lambda, -\lambda}$ is $2$-connected (see Theorem \ref{opposite-eigen-values-are-2-connected}).

\begin{lemma}\label{B}
Let $A=\begin{pmatrix}
a & b\\ c & d
\end{pmatrix} \in M_{1, -1}$ such that $a \notin \{1, -1\}$. Then there exist $X, Y\in M_{1, -1}$ such that $D(1,-1)*X=A$ and $D(-1,1) *Y=A$.
\end{lemma}
\begin{proof}
Since $A \in M_{1,-1}$ and $a \notin \{1, -1\}$, it implies that $d \notin \{1, -1\}$, and $b$ and $c$ are non-zero. Thus if we substitute these values in Equation \eqref{eq10}, the value of $v$ cannot be zero unless $u$ is zero. Thus, we can use the same type of proof as done in Lemma \ref{A}.
\end{proof}

\begin{lemma}\label{M(1,-1) is not 1-connected}
Let $A=\begin{pmatrix}
1 & 0 \\
0 & -1
\end{pmatrix}$ and $B=\begin{pmatrix}
-1 & 0\\
\alpha &1 
\end{pmatrix}$, where $\alpha \in \mathbb{C}^*$. Then there does not exist any $X \in M_{1,- 1}$ such that $A*X=B$.
\end{lemma}
\begin{proof}
Assume on the contrary and suppose that there exists $X \in M_{1,- 1}$ with
$X =\begin{pmatrix} a & b\\c & d \end{pmatrix}$, $a, b , c, d \in \mathbb{C}$, $a+d=0$, $ad-bc=-1$ and
$$
\begin{pmatrix} 1 & 0 \\ 0 & -1 \end{pmatrix}  \begin{pmatrix} a & b\\c & d \end{pmatrix} 
=
\begin{pmatrix} a & b\\c & d \end{pmatrix}  \begin{pmatrix} -1 & 0\\ \alpha &1 \end{pmatrix}.
$$
This implies that $d=a=b=0$, which contradicts $ad-bc=-1$. This completes the proof.
\end{proof}

\begin{lemma}\label{E}
Let $\gamma=\pm 1$ and $\alpha \in \mathbb{C}$. Then the following holds:
\begin{enumerate}
\item there exists $X \in M_{1, -1}$ such that $D(\gamma, -\gamma) *X=\begin{pmatrix}
\gamma & \alpha\\
0 & -\gamma
\end{pmatrix}$.
\item there exists $X \in M_{1, -1}$ such that $D(\gamma, -\gamma) *X=\begin{pmatrix}
\gamma & 0\\
\alpha & -\gamma
\end{pmatrix}$.
\end{enumerate}
\end{lemma}
\begin{proof}
We will prove {\it (1)} part. The proof of {\it (2)} is along similar lines.
\par 

Suppose $A=\begin{pmatrix}
\gamma & \alpha\\
0 & -\gamma
\end{pmatrix}$ and $X_{\beta} =\begin{pmatrix}
\gamma & \beta\\
0 & -\gamma
\end{pmatrix}$. Then, 
$$\begin{pmatrix}
  \gamma & \beta \\
  0 & -\gamma
\end{pmatrix}^{-1}\begin{pmatrix}
  \gamma & 0 \\
  0 & -\gamma
\end{pmatrix}\begin{pmatrix}
  \gamma & \beta \\
  0 & -\gamma
\end{pmatrix} =
\begin{pmatrix}
\gamma & 2\beta \\
0 & -\gamma
\end{pmatrix}$$ which is an upper triangular matrix in $\M_{1, -1}$.
Thus taking $\beta=\alpha/2$ implies that $D(\gamma,-\gamma) * X_{\beta}=A.$
\end{proof}

\begin{theorem}\label{M(1, -1) is 2-connected}
The quandle $M_{1, -1}$ is a $2$-connected quandle.
\end{theorem}
\begin{proof}
The quandle $M_{1, -1}$ is non-trivial (see Lemma \ref{non-trivial-subquandles}). By Lemma \ref{M(1,-1) is not 1-connected}, we know that $M_{1, -1}$ is not $1$-connected. Noting that if $X \in M_{1,-1}$, then $X^{-1} \in M_{1,-1}$. By Theorem \ref{sufficinet-condition-for-n-connected}, we know that to prove $M_{1,-1}$ is $2$-connected, it is sufficient to prove that for any $A \in M_{1, -1}$, either there exists $X\in M_{1, -1}$ such that $D(1,-1)*X=A$ or there exists $Y, Z \in M_{1, -1}$ such that $D(1,-1) *Y*Z=A$. In the following points we have covered all the possibilities for $A\in M_{1,-1}$.
\begin{enumerate}[label=(\alph*)]
\item If $A=D(-1,1)$, then by Lemma \ref{joining D(1, -1) to D(-1,1)} there exists $X \in M_{1,-1}$ such that $D(1,-1)*X=D(-1,1)=A$. \label{casea}
\item If $A\in \left\{\begin{pmatrix} 1 & \alpha\\ 0 & -1 \end{pmatrix}, \begin{pmatrix} 1 & 0\\ \alpha & -1 \end{pmatrix} \right\}$, where $\alpha \in \mathbb{C}$, then by Lemma \ref{E}, there exists $X\in M_{1, -1}$ such that $D(1,-1) * X=A$. \label{caseb}
\item If $A \in \left\{\begin{pmatrix}
-1 & \alpha \\ 0 & 1
\end{pmatrix}, \begin{pmatrix}
-1 & 0 \\ \alpha & 1
\end{pmatrix}\right\}$, where $\alpha\in \mathbb{C}$, then by Lemma \ref{E} there exists $Z \in M_{1,-1}$ such that $D(-1,1) *Z=A$. Now by Case \ref{casea} we see that $D(1,-1)*Y*Z=A$ for some $Y \in M_{1,-1}$.
\item Let $A=\begin{pmatrix}
a & b \\c & d 
\end{pmatrix}$, where $a \notin \{1, -1\}.$ Then by Lemma \ref{B}, there exists $X \in M_{1, -1}$ such that $D(1,-1) *X=A$.
\end{enumerate}
%%%%%%%%%%%%%%%%%%%%%%%%%%%%%%%%%%%%%%%%%%%%%%%%%%%%%%%%%%%%%%%
This completes the proof.
\end{proof}

\begin{theorem}\label{f}
Let $\lambda, \lambda_1, \lambda_2, k\in \mathbb{C}^*$.
Then, $\M_{\lambda_1, \lambda_2}$ and $\M_{k\lambda_1, k\lambda_2}$ are isomorphic quandles.
\end{theorem}
\begin{proof}
Set
\begin{align*}
f_k\, : \, \M_{\lambda_1, \lambda_2} &\longrightarrow \M_{k\lambda_1, k\lambda_2}\\
        A& \longmapsto kA
\end{align*}
We leave the proof of injectivity and surjectivity to the reader and we prove that $f_k$ is a quandle homomorphism: 
\begin{align*}
f_k(A*B)&=f_k(B^{-1}AB)\\
&= k(B^{-1}AB)\\
&= B^{-1}k^{-1}kAkB\\
&= (kB)^{-1}(kA)(kB)\\
&=\big(f(B)\big)^{-1}f(A)f(B)\\
&=f(A)*f(B) .
\end{align*}
\end{proof}

\begin{remark}\label{remark:f}
The proof of Theorem \ref{f} may suggest that the same argument is valid with $k$ replaced by $K$, a matrix in the centralizer of  quandle $M_{\lambda_1, \lambda_2}$ (viewed as a subset of $\GL(2, \mathbb{C})$). We will next show that, although true, this does not contribute any new information. For $\lambda\in \mathbb{C}^*$, $M_{\lambda, \lambda}$ has only one element. Thus, the $M_{\lambda, \lambda}$'s are isomorphic among themselves and non-isomorphic with $M_{\lambda}$ and $M_{\lambda_1, \lambda_2}$, where $\lambda_1 \neq \lambda_2 \in \mathbb{C}^*$, since the latter are infinite quandles. Given $\lambda_1 \neq \lambda_2 \in \mathbb{C}^*$ and $K \in \GL(2, \mathbb{C})$,  for $K$ to commute with $D(\lambda_1, \lambda_2)$ then $K$ has to be a diagonal matrix. To further  commute with other matrices, $K$ also has to be a scalar matrix. But this amounts to multiplication by a scalar. This concludes this remark.
\end{remark}

\begin{corollary}\label{f.1}
Let $\lambda, k \in \mathbb{C}^*$. Then $M_{\lambda}$ and $M_{k\lambda}$ are isomorphic quandles.
\end{corollary}
\begin{proof}
Analogous to the proof of Theorem \ref{f}.
\end{proof}

\begin{remark}
Analogous to Remark \ref{remark:f} we note down that in Corollary \ref{f.1} it is of no use to replace $k$ by a matrix $K$ which is in the centralizer of $M_{\lambda}.$
\end{remark}

\begin{theorem}\label{opposite-eigen-values-are-2-connected}
The quandle $M_{\lambda, -\lambda}$ is a $2$-connected quandle.
\end{theorem}
\begin{proof}
The result follows from Theorem \ref{M(1, -1) is 2-connected} and Corollary \ref{f}.
\end{proof}

\subsection{The proof that $M_{\lambda}$ is $2$-connected}
In this subsection, we will prove that for $\lambda \in \mathbb{C}^*$, the quandle $M_{\lambda}$ is $2$-connected (see Theorem \ref{same-eigen-value-is-2-connected}).

\begin{lemma}\label{M1-not-1-connected}
The quandle $M_{1}$ is not $1$-connected.
\end{lemma}
\begin{proof}
Let $\begin{pmatrix}
a & b \\ c& d
\end{pmatrix} \in \GL(2, \mathbb{C})$, then
\begin{align*}
\begin{pmatrix} 1 & 2 \\ 0 & 1 \end{pmatrix} \begin{pmatrix} a & b \\ c & d \end{pmatrix}
&=
\begin{pmatrix} a & b \\ c & d \end{pmatrix} \begin{pmatrix} 1 & 1 \\  0 & 1 \end{pmatrix}\\
\iff \begin{pmatrix} a+ 2c & b+2d \\ c & d \end{pmatrix} 
&=
\begin{pmatrix} a & a+b \\ c & c+d \end{pmatrix}
\end{align*}
implies that $d=a/2$ and $c=0$. Now for $X$ to be in $M_1$, the trace  criteria implies that $a=4/3$, but the determinant criteria implies that $a=\pm \sqrt{2}. $ Thus $X \notin M_1$, and hence $M_1$ is not $1$-connected.
\end{proof}

\begin{lemma}\label{2-connectedness-in-M1}
Let $A =\begin{pmatrix}
a & b \\ c& d
\end{pmatrix}\in M_1$.
\begin{enumerate}
\item If $ b \neq 0$, then there exists $X \in M_1$ such that $X^{-1} \begin{pmatrix}
1 & 0 \\1 & 1
\end{pmatrix}X =A$.
\item If $c \neq 0$, then there exists $X \in M_1$ such that $X^{-1} \begin{pmatrix}
1 & 1 \\0 & 1
\end{pmatrix}X =A$.
\end{enumerate}
\end{lemma}
\begin{proof}
We will prove the first one, the other is similar to it.
\par 

Let $y \neq 0$, and $\begin{pmatrix}
x & y \\ z & w \end{pmatrix} \in M_1$. Then
\begin{align*}
&\begin{pmatrix} x & y \\ z & w \end{pmatrix} \begin{pmatrix} \alpha & \beta \\ \gamma & \delta \end{pmatrix}
=
\begin{pmatrix} 1 & 0 \\ 1 &1 \end{pmatrix} \begin{pmatrix} x & y \\ z & w \end{pmatrix}\\
\implies & \begin{cases}
x\alpha + y\gamma =x\\
x\beta + y\delta =y\\
z\alpha + w\gamma =x+z\\
z\beta + w\delta =y + w
\end{cases}
\end{align*}

$
\begin{pmatrix}
x & 0 & y & 0\\
0 & x & 0 & y\\
z & 0 & w & 0\\
0 & z & 0 & w
\end{pmatrix} \begin{pmatrix} \alpha \\ \beta \\\gamma \\\delta \end{pmatrix}
=
\begin{pmatrix}
x \\ y \\ x+z \\ y +w
\end{pmatrix}
$

One can check that the matrix  $\mathcal{X}= \begin{pmatrix}
x & 0 & y & 0\\
0 & x & 0 & y\\
z & 0 & w & 0\\
0 & z & 0 & w
\end{pmatrix}$ is invertible ($\det \mathcal{X}=1$), thus we have a unique solution for $\alpha, \beta , \gamma$ and $\delta$ in terms of $x, y, z, w$. Thus

\[
 \begin{pmatrix} \alpha \\ \beta \\\gamma \\\delta \end{pmatrix}
=\begin{pmatrix}
w & 0 & -y & 0\\
0 & w & 0 & -y\\
-z & 0 & x & 0\\
0 & -z & 0 & x
\end{pmatrix}
\begin{pmatrix}
x \\ y \\ x+z \\ y +w
\end{pmatrix}
\]
which implies 
\[\begin{cases}
\alpha = 1-xy\\
\beta=-y^2\\
\gamma=x^2\\
\delta=1+xy
\end{cases} \]
Thus for any $A=\begin{pmatrix}
a & b \\ c & d
\end{pmatrix} \in M_1$ such that $b \neq 0$, there exists $\begin{pmatrix}
x & y \\ z & w
\end{pmatrix} \in M_1$, where either $x= \sqrt{c}, ~y= \iota \sqrt{b}, ~z=\dfrac{2 \sqrt{c}-c -1}{ \iota \sqrt{b}}~,w=2 - \sqrt{c}$  or $x= \sqrt{c}, ~y=- \iota \sqrt{b}, ~z=\dfrac{ 2 \sqrt{c}-c -1}{- \iota \sqrt{b}}~,w=2 - \sqrt{c}$, such that $X^{-1} \begin{pmatrix} 1 & 0\\ 1& 1\end{pmatrix} X=A$.
\end{proof}

\begin{theorem}\label{M1-is-2-connected}
The quandle $M_1$ is $2$-connected.
\end{theorem}
\begin{proof}
By Lemma \ref{M1-not-1-connected}, we know $M_1$ is not $1$-connected. By Lemma \ref{2-connectedness-in-M1}, we know that for any $A \in M_1$, either there exists $X \in M_1$ such that $\begin{pmatrix}
1 & 0\\1 & 1
\end{pmatrix} *X=A$ or there exist $Y, Z \in M_1$ such that $\begin{pmatrix}
1 & 0\\1 & 1
\end{pmatrix} *Y*Z=A$. Thus by Theorem \ref{sufficinet-condition-for-n-connected}, $M_1$ is $2$-connected.
\end{proof}

\begin{theorem}\label{same-eigen-value-is-2-connected}
Let $\lambda \in \mathbb{C}^*$. Then $M_{\lambda}$ is $2$-connected.
\end{theorem}
\begin{proof}
The result follows from Theorem \ref{2-connectedness-in-M1} and Corollary \ref{f.1}.
\end{proof}

%%%%%%%%%%%%%%%%%%%%%%%%%%%%%%%%%%%%%%%%%%%%%%%%%%%%%%%%%%%%%%%%%
%%%%%%%%%%%%%%%%%%%%%%%%%%%%%%%%%%%%%%%%%%%%%%%%%%%%%%%%%%%%%%%%%
%%%%%%%%%%%%%%%%%%%%%%%%%%%%%%%%%%%%%%%%%%%%%%%%%%%%%%%%%%%%%%%%%

\section{Study of $\Conj(\GL(2,\mathbb{C}))$}\label{sec:evidence-for-infinitely-number-of-quandles}
In this section we will prove the following theorems.

\begin{Theorem}
Let $ n \geq 2$. Then $\Conj(\GL(2,\mathbb{C})) \ncong \Conj_n(\GL(2,\mathbb{C}))$.
\end{Theorem}

\begin{Theorem}
For any $m, n \in \mathbb{Z}^+$ with $m \neq n$, the quandle $\Conj_m(\PGL(2, \mathbb{C}))$ is not isomorphic to $\Conj_n(\PGL(2, \mathbb{C}))$.
\end{Theorem}

\begin{proposition}
Let $\Cl(A)=\{P^{-1} A P~|~P\in \GL(2, \mathbb{C})\}$ and $\Cl(A,n)=\{P^{-n} A P^n~|~P\in \GL(2, \mathbb{C})\}$, where $n \in \mathbb{Z}^+$ and $A \in \GL(2, \mathbb{C})$. Then $\Cl(A)=\Cl(A,n)$.
\end{proposition}
\begin{proof}
One way is obvious. Let $P^{-1} A P \in \Cl(A)$. Since $P$ is invertible, we can find a matrix $P' \in \GL(2,\mathbb{C})$ such that $P'^n=P$ (see \cite[Theorem 4.1]{MR2069275}). Thus we have $P^{-1} A P=P'^{-n} A P'^n$.
\end{proof}

\begin{proposition}\label{connected-components-of-higher-quandles-of-GL(2,C)}
For any $n \in \mathbb{Z}^+$, there are exactly three families of connected components in $\Conj_n(\GL(2, \mathbb{C}))$, namely:
\begin{enumerate}
\item for each $\lambda_1 \neq \lambda_2 \in \mathbb{C}^*$, the conjugacy class of 
$\begin{pmatrix}
  \lambda_1 & 0 \\
  0 & \lambda_2
\end{pmatrix}$ in $\GL(2,\mathbb{C})$,
\item for each  $\lambda \in \mathbb{C}^*$, the conjugacy class of 
$\begin{pmatrix}
  \lambda & 1 \\
  0 & \lambda
\end{pmatrix}$
in $\GL(2, \mathbb{C})$,
\item for each $\lambda \in \mathbb{C}^*$, the singleton set $\left \{
\begin{pmatrix}
  \lambda &0 \\
  0 & \lambda
\end{pmatrix} \right \}$.
\end{enumerate}
\end{proposition}
\begin{proof}
Proof follows from Proposition \ref{connected-components-of-higher-quandles-of-GL(2,C)} and considerations written down in Section \ref{sec:connected-components-of-GL(2,C)}.
\end{proof}

\begin{lemma}\label{M_lambda_non-trivial-subquandle}
Let $\lambda \in \mathbb{C}^*$ and $n \in \mathbb{Z}^+$. Then $\type(M_{\lambda})=\infty$, where $M_{\lambda}$ is a subquandle of $\Conj_n(\GL(2, \mathbb{C})).$
\end{lemma}
\begin{proof}
Let $P=\begin{pmatrix}
\lambda & 0\\
1 & \lambda
\end{pmatrix}$. By induction one can prove that $P^n=\begin{pmatrix}
\lambda^n & 0\\
n\lambda^{n-1} & \lambda^n
\end{pmatrix}$ for all $n \in \mathbb{Z}^+$. Now,
\begin{align*}
\begin{pmatrix}
\lambda & 1\\
0 & \lambda
\end{pmatrix} \underbrace{*_n P *_n \cdots *_n P}_{m ~\textrm{times}} &=
P^{-nm} \begin{pmatrix}
\lambda & 1\\
0 & \lambda
\end{pmatrix} P^{nm}\\
&= \begin{pmatrix}
\lambda+ (nm)/\lambda & 1\\
-(nm/\lambda)^2 & \lambda -(nm)/\lambda
\end{pmatrix}\\
&\neq \begin{pmatrix}
\lambda & 1 \\
0 & \lambda
\end{pmatrix}.
\end{align*}
Thus, $\type{M_{\lambda}}=\infty$.
\end{proof}

%\iffalse{\begin{corollary}
%For $\lambda \in \mathbb{C}^*$, the connected component $M_{\lambda}$ is a non-trivial subquandle of $\Conj_n(\GL(2, \mathbb{C}))$.
%\end{corollary}}\fi

\begin{corollary}
For each $n \in \mathbb{Z}^+$, $\type(\Conj_n(\GL(2, \mathbb{C}))=\infty$.
\end{corollary}

\begin{definition}
An element $z \in \mathbb{C}$ is said to be a {\it root of unity} if there exists $n \in \mathbb{Z}^+$ such that $z^n=1$. In this case, we also say that $z$ is an {\it $n$-th root of unity}.
\end{definition}

\begin{definition}
For $p \in \mathbb{Z}^+$, an element $z \in \mathbb{C}$ is said to be a {\it $p$-th primitive root of unity} if $z^p=1$ and $z^k\neq 1$ for $1 \leq k \leq p-1$.
\end{definition}

\begin{lemma}\label{n-subquandles-in-GL}
Let $\omega_1$ and $\omega_2$ be $n$-th roots of unity, and $\lambda \in \mathbb{C}^*$. Then $M_{\lambda \omega_1, \lambda \omega_2}$ is an $n$-subquandle of $\Conj(\GL(2,\mathbb{C}))$.
\end{lemma}
\begin{proof}
Let $X \in M_{\lambda \omega_1, \lambda \omega_2}$. Then there exists $P \in \GL(2, \mathbb{C})$ such that
$$P^{-1} \begin{pmatrix}
\lambda \omega_1 & 0\\
0 & \lambda \omega_2
\end{pmatrix} P=X.$$
Since $X ^n=\lambda^{n} \mathbb{I}$, then $A\underbrace{* X * \cdots *X}_{n-times}=X^{-n} A X^{n}=A$, for all $A, X \in M_{\lambda \omega_1, \lambda \omega_2}$.
\end{proof}

\begin{corollary}\label{lemma-of-n-subquandles-in-GL}
Let $\omega_1$ and $\omega_2$ be $n$-th roots of unity, and $\lambda \in \mathbb{C}^*$. Then $M_{\lambda \omega_1, \lambda \omega_2}$ is a trivial subquandle of $\Conj_n(\GL(2,\mathbb{C}))$.
\end{corollary}

\begin{lemma}\label{condition for n-quandle}
Let $\lambda_1, \lambda_2 \in \mathbb{C}^*$, and $n \in \mathbb{Z}^+$. Then $M_{\lambda_1, \lambda_2}$  is an $n$-subquandle of $\Conj(\GL(2,\mathbb{C}))$ if and only if $\lambda_1/\lambda_2$ is an $n$-th root of unity.
\end{lemma}
\begin{proof}
If $\lambda_1=\lambda_2$, then $M_{\lambda_1, \lambda_2}$ being a trivial subquandle implies that it is an $n$-subquandle.
\par 

If $\lambda_1\neq \lambda_2$ and $\lambda_1/\lambda_2$ is an $n$-th root of unity, then by Lemma \ref{n-subquandles-in-GL}, $M_{\lambda_1, \lambda_2}$ is an $n$-subquandle of $\Conj(\GL(2,\mathbb{C}))$.
\par 
Now we will prove that for $\lambda_1, \lambda_2 \in \mathbb{C}^*$ where $\lambda_1 \neq \lambda_2$, if $M_{\lambda_1, \lambda_2}$ is an $n$-subquandle of $\Conj(\GL(2, \mathbb{C}))$, then $\lambda_1/\lambda_2$ is an $n$-th root of unity. Let $X=\begin{pmatrix}
\lambda_1 & \alpha\\
0& \lambda_2
\end{pmatrix}$ and $P=\begin{pmatrix}
\lambda_1 & 0\\
0 & \lambda_2
\end{pmatrix}$, where $\alpha\in \mathbb{C}^*$. Clearly, $X$ and $ P $ are in $M_{\lambda_1, \lambda_2}$. By induction we have $P^n=\begin{pmatrix}
\lambda_1^n & 0\\
0 & \lambda_2^n
\end{pmatrix}$.  Now

\begin{align*}
X \underbrace{*P * \cdots * P}_{n-\textrm{times}}  &=P^{-n} X P^n \\
&= (1/(\lambda_1 \lambda_2)^n) \begin{pmatrix}
\lambda_2^{n} & 0\\
0 & \lambda_1^n
\end{pmatrix}
\begin{pmatrix}
\lambda_1 & \alpha\\
0& \lambda_2
\end{pmatrix}
\begin{pmatrix}
\lambda_1^n & 0\\
0 & \lambda_2^n
\end{pmatrix}\\
&=\begin{pmatrix}
\lambda_1 & \alpha (\lambda_2/\lambda_1)^n\\
0 & \lambda_2
\end{pmatrix}.
\end{align*}

For $M_{\lambda_1, \lambda_2}$ to be $n$-subquandle of $\Conj(\GL(2, \mathbb{C})$, the equality $\alpha (\lambda_2 /\lambda_1)^n=\alpha$ must hold for all $\alpha\in \mathbb{C}^*$, which implies that $(\lambda_1 /\lambda_2)^n=1$.

\end{proof}

\begin{corollary}
Let $\lambda_1, \lambda_2 \in \mathbb{C}^*$, and $m \in \mathbb{Z}^+$. Then the type of $M_{\lambda_1, \lambda_2}$  as a subquandle of $\Conj_m(\GL(2,\mathbb{C}))$ is finite if and only if $\lambda_1/\lambda_2$ is a $k$-th root of unity for some $k \in \mathbb{Z}^+$.
\end{corollary}

\iffalse{\begin{corollary}
Let $\omega_1$ and $\omega_2$ are distinct $n$-th roots of unity, and $\lambda \in \mathbb{C}^*$. If $m_1$ and $m_2$ are the least positive integers such that $\omega_1^{m_1}=\omega_2^{m_2}=1$, then $\type(M_{\lambda \omega_1, \lambda \omega_2})=lcm(m_1,m_2)$, where $M_{\lambda \omega_1, \lambda \omega_2}$ is a subquandle of $\Conj(\GL(2,\mathbb{C}))$.
\end{corollary}}\fi

Now we prove the first main theorem.
\begin{theorem}\label{first-main-result-of-mutliplying-quandles}
The quandle $\Conj(\GL(2, \mathbb{C}))$ is not isomorphic to $\Conj_n(\GL(2, \mathbb{C}))$ for any $n \geq 2$.
\end{theorem}
\begin{proof}
Let us suppose there exists an isomorphism $$f: \Conj(\GL(2, \mathbb{C}) \to \Conj_n(\GL(2, \mathbb{C})).$$ Then the image of each connected component under $f$ must be a connected component and the type of each connected component must be preserved under the map $f$. This implies that each connected component of $\Conj_n(\GL(2, \mathbb{C}))$, $n \geq 2$ having more than one element is a non-trivial subquandle (see Lemma \ref{non-trivial-subquandles}), which is not true by Corollary \ref{lemma-of-n-subquandles-in-GL} (see Proposition \ref{connected-components-of-higher-quandles-of-GL(2,C)}). Thus $\Conj(\GL(2, \mathbb{C})) \ncong \Conj_n(\GL(2, \mathbb{C}))$ for any $n \geq 2$.
\end{proof}

For $A \in \GL(2, \mathbb{C})$, we denote its image in $\PGL(2, \mathbb{C})$ by $[A]$,  that is, $[A]=\{kA~|~k \in \mathbb{C}^*\}$.

\begin{remark}
Let $[A], [B] \in \PGL(2, \mathbb{C})$. Then $[A]$ and $[B]$ are in the same conjugacy class if and only if $A$ is conjugate to $kB$ in $\GL(2, \mathbb{C})$ for some $k \in \mathbb{C}^*$. This implies that following are the only families of  distinct conjugacy classes in $\PGL(2, \mathbb{C})$:
\begin{enumerate}
\item The set $M_{[\lambda_1, \lambda_2]}=\{[A]~|~A \in M_{k\lambda_1, k \lambda_2}~\textrm{ for some }~k \in \mathbb{C}^*\}$, where $\lambda_1$ and $\lambda_2$ are distinct elements in $\mathbb{C}^*$. \\
\item $M_{[\lambda]}=\{[A]~|~A \in M_{k\lambda}~\textrm{ for some }~k \in \mathbb{C}^*\}$, where $\lambda \in \mathbb{C}^*$.\\
\item the singleton set containing the identity element of $\PGL(2, \mathbb{C})$.
\end{enumerate}
\end{remark}

We note down the following results for $\PGL(2, \mathbb{C})$ without writing their proofs as they can be proved along similar lines as done in the case of  $\GL(2, \mathbb{C})$, see above in this section.

\begin{proposition}
For any $n \in \mathbb{Z}^+$, the only connected components in $\Conj_n(\PGL(2, \mathbb{C}))$ are $M_{[\lambda_1, \lambda_2]}$, $M_{[\lambda]}$ and the identity element, where $\lambda_1, \lambda_2, \lambda \in \mathbb{C}^*$ and $\lambda_1 \neq \lambda_2$.
\end{proposition}

\begin{lemma}\label{M_[lambda]_non-trivial-subquandle}
Let $\lambda \in \mathbb{C}^*$ and $n \in \mathbb{Z}^+$. Then $\type(M_{[\lambda]})=\infty$, where $M_{[\lambda]}$ is a subquandle of $\Conj_n(\PGL(2,\mathbb{C}))$.
\end{lemma}

\begin{corollary}
For each $n \in \mathbb{Z}^+$, $\type(\Conj_n(\PGL(2, \mathbb{C}))=\infty$.
\end{corollary}
\begin{lemma}\label{non-trivial-subquandles-PGL}
Each connected component of $\Conj(\PGL(2,\mathbb{C}))$ having more than one element is a non-trivial subquandle.
\end{lemma}

\iffalse{\begin{lemma}
Let $\omega_1$ and $\omega_2$ are distinct $n$-th roots of unity. Then $M_{[\omega_1, \omega_2]}$ is an $n$-subquandle of $\Conj(\PGL(2,\mathbb{C}))$.
\end{lemma}}\fi

\begin{lemma}\label{n-subquandles-PGL}
Let $\lambda_1$ and $ \lambda_2$ be in $\mathbb{C}^*$, and $n \in \mathbb{Z}^+$. Then $M_{[\lambda_1, \lambda_2]}$  is an $n$-subquandle of $\Conj(\PGL(2,\mathbb{C}))$ if and only if $\lambda_1/\lambda_2$ is an $n$-th root of unity.
\end{lemma}

\iffalse{\begin{corollary}
Let $\omega_1$ and $\omega_2$ are distinct $n$-th roots of unity. If $m_1$ and $m_2$ are the least positive integers such that $\omega_1^{m_1}=\omega_2^{m_2}=1$, then $\type(M_{[ \omega_1, \omega_2]})=lcm(m_1,m_2)$, where $M_{[ \omega_1, \omega_2]}$ is a subquandle of $\Conj(\PGL(2,\mathbb{C}))$.
\end{corollary}

\begin{lemma}%\label{counting-trivial-subquandle-in-PGL}
The quandle $\Conj_n(\PGL(2, \mathbb{C}))$ has exactly $(\binom{n}{2}+1)$ connected components which are trivial subquandles.
\end{lemma}
\begin{proof}
The connected components in $\Conj_n(\PGL(2, \mathbb{C}))$ which are trivial as quandle are exactly those connected connected components in $\Conj(\PGL(2, \mathbb{C}))$ which are $n$-subqaundles. The number of ways one can choose two distinct $n$-th roots of unity is $\binom{n}{2}$. Thus, by Lemma \ref{n-subquandles-PGL} there are only $\binom{n}{2}+1$ connected components in $\Conj(PGL(2, \mathbb{C}))$ which are $n$-subquandles.
\end{proof}
}\fi

\begin{lemma}\label{counting-trivial-subquandle-in-PGL}
The quandle $\Conj_n(\PGL(2, \mathbb{C}))$ has exactly $n$ connected components which are trivial subquandles.
\end{lemma}
\begin{proof}
The connected components in $\Conj_n(\PGL(2, \mathbb{C}))$ which are trivial are exactly those connected components in $\Conj(\PGL(2, \mathbb{C}))$ which are $n$-subquandles. The number of ways one can choose two $n$-th roots of unity such that one of them is identity is $n$. Thus, by Lemma \ref{n-subquandles-PGL} there are only $n$ connected components in $\Conj(\PGL(2, \mathbb{C}))$ which are $n$-subquandles.
\end{proof}
Thus we have the second main theorem.
\begin{theorem}\label{the-main-theorem-of multiplying-quandles}
For any $m, n \in \mathbb{Z}^+$ with $m \neq n$, the quandle $\Conj_m(\PGL(2, \mathbb{C}))$ is not isomorphic to $\Conj_n(\PGL(2, \mathbb{C}))$.
\end{theorem}
\begin{proof}
The result follows from Lemma \ref{counting-trivial-subquandle-in-PGL}.
\end{proof}

\begin{question}
Let $m, n \in \mathbb{Z}^+$ such that $m \neq n$. Is it true that $\Conj_m(\GL(2, \mathbb{C})) \ncong \Conj_n(\GL(2, \mathbb{C}))$?
\end{question}

\section{Classification of connected components of $\Conj(\GL(2, \mathbb{C})$}\label{classification-of-connected-components}

In Section \ref{sec:connected-components-of-GL(2,C)}, we proved that the non-trivial connected components of $\Conj(\GL(2, \mathbb{C}))$ are $2$-connected (see Theorem \ref{distinct-pm-eigen-values-are-2-connected}, \ref{opposite-eigen-values-are-2-connected}, \ref{same-eigen-value-is-2-connected}), and hence they are connected subquandles. In this section, we partially classify these connected components.

\begin{theorem}\label{thm:6.1}
Let $\lambda, \lambda_1, \lambda_2 \in \mathbb{C}^*$. Then $M_{\lambda}$ is not isomorphic to $M_{\lambda_1, \lambda_2}$.
\end{theorem}
\begin{proof}
Assume on the contrary and suppose that $f:M_{\lambda} \to M_{\lambda_1, \lambda_2}$ is a quandle isomorphism. Since $\begin{pmatrix}
\lambda& 1\\0 &\lambda
\end{pmatrix}$, $\begin{pmatrix}
\lambda & 2\\0 &\lambda
\end{pmatrix}$ and $\begin{pmatrix}
\lambda & 3\\0 & \lambda
\end{pmatrix}$ commute with one another, so do their images under the map $f$ (as $f$ is a quandle homomorphism). Thus there exists $S \in \GL(2, \mathbb{C})$ which simultaneously diagonalize $f(\begin{pmatrix}
\lambda& 1\\0 &\lambda
\end{pmatrix})$, $f(\begin{pmatrix}
\lambda & 2\\0 &\lambda
\end{pmatrix})$, and $f(\begin{pmatrix}
\lambda & 3\\0 & \lambda
\end{pmatrix})$ (see \cite[Page 569]{MR1878556}). But since there are only two diagonal matrices in $M_{\lambda_1, \lambda_2}$ and $f$ is a bijection, this is absurd. The proof is complete.
\end{proof}

\begin{theorem}\label{thm:6.2}
Let $\lambda_1, \lambda_2, \lambda_1', \lambda_2' \in \mathbb{C}$ and $p, q \in \mathbb{Z}^+$ where $p \neq q$. Then the following holds:
\begin{enumerate}
\item If $\lambda_2/\lambda_1$ is a $p$-th primitive root of unity and $\lambda_2'/\lambda_1'$ is not a root of unity, then $M_{\lambda_1, \lambda_2}$ is not isomorphic to $M_{\lambda_1',\lambda_2'}$.
\item If $\lambda_2/\lambda_1$ is $p$-th primitive root of unity and $\lambda_2'/\lambda_1'$ is $q$-th primitive root of unity, then $M_{\lambda_1, \lambda_2}$ is not isomorphic to $M_{\lambda_1',\lambda_2'}$.
\end{enumerate}
\end{theorem}
\begin{proof}
The proof of {\it (1)} follows from Lemma \ref{condition for n-quandle}.
\par 

Now for the proof of {\it (2)}. Without loss of generality we assume $q < p$. By Lemma \ref{condition for n-quandle} and Theorem \ref{f}, it is sufficient to prove that $M_{1, \omega}$ is not isomorphic to $M_{1, \omega'}$, where $\omega$ and $\omega'$ are $p$-th and $q$-th primitive roots of unity, respectively. Note that $M_{1, \omega}$ is $p$-quandle and $M_{1, \omega'}$ is $q$-quandle. Now let $X=\begin{pmatrix} 1 &0\\ 0 & w \end{pmatrix}$ which is an element in $M_{1, \omega}$. Then
\begin{align*}
\begin{pmatrix} 1 & 1 \\ 0 & \omega \end{pmatrix} *^q X&=X^{-q} \begin{pmatrix} 1 & 1 \\ 0 & \omega \end{pmatrix} X^q\\
&=\begin{pmatrix} 1 & w^q\\ 0 & w \end{pmatrix}\\
&\neq \begin{pmatrix} 1 & 1 \\ 0 & \omega \end{pmatrix},
\end{align*}
as $q <p$ and $\omega$ is a $p$-th primitive root of unity. Thus $M_{1, \omega}$ is not a $q$-quandle and hence $M_{1, \omega }$ is not isomorphic to $M_{1, \omega'}$.
\end{proof}
\begin{lemma}\label{higher-isomorphisms}
Let $\lambda, \lambda' \in \mathbb{C}$ and $M_{1, \lambda}$ be quandle isomorphic to $M_{1, \lambda'}$. If $\lambda$ is not a root of unity, then $M_{1, \lambda^n}$ is isomorphic to $M_{1, \lambda'^n}$ for all $n \in \mathbb{Z}^+$.
\end{lemma}
\begin{proof}
Let $h_1: M_{1, \lambda} \to M_{1, \lambda'}$ be a quandle isomorphism. Now for $n \in \mathbb{Z}^+$, define \begin{align*}
h_n: M_{1, \lambda^n} &\to M_{1, \lambda'^n}~ \textrm{ as }\\
B&\mapsto h_1(A)^n,
\end{align*}
where $B \in M_{1, \lambda^n}$, $B=A^n$, $A \in M_{1, \lambda}$. Note that such $A$ always exists, see \cite{MR2069275}.
\par 

Now we will prove that $h_n's$ are well-defined and are bijective. Suppose $A_1, A_2 \in M_{1, \lambda}$ such that $A_1^n =A_2^n$. Then there exist $P_1, P_2 \in \GL(2, \mathbb{C})$ such that $A_1=P_1^{-1} D(1, \lambda) P_1$ and $A_2=P_2^{-1} D(1, \lambda) P_2$. Then $A_1^n=A_2^n$ implies that
\begin{align*}
(P_1^{-1} D(1, \lambda) P_1)^n&=(P_2^{-1} D(1, \lambda) P_2)^n\\
P_1^{-1} D(1, \lambda^n) P_1&=P_2^{-1} D(1, \lambda^n) P_2\\
(P_1P_2^{-1})^{-1}D(1, \lambda^n) (P_1P_2^{-1})&=D(1, \lambda^n)\\
\big((P_1P_2^{-1})^{-1}D(1, \lambda) (P_1P_2^{-1})\big)^n&=D(1, \lambda^n),
\end{align*}
which further implies that $(P_1P_2^{-1})^{-1}D(1, \lambda) (P_1P_2^{-1})=D(\omega, \omega'\lambda)$ (see Theorem \cite[Theorem 4.1]{MR2069275} and noting that any $2$-by-$2$ matrix with distinct non-zero eigen values have four square roots) where $\omega$ and $\omega'$ are $n$-th roots of unity. Noting that under similar transformation, eigen values are preserved and $\lambda$ is not a root of unity, thus $(P_1P_2^{-1})^{-1}D(1, \lambda) (P_1P_2^{-1})=D(1, \lambda)$ and hence $A_1=A_2$. Thus each matrix in $M_{1, \lambda^n}$ has a unique root in $M_{1, \lambda}$, and as a consequence the maps $h_n:M_{1, \lambda^n} \to M_{1, \lambda'^n}$ are well-defined for all $n \in \mathbb{Z}^+$. Furthermore, since $h_1$ is a quandle isomorphism, thus $\lambda'$ is not a root of unity, and therefore by the preceding kind of reasoning, the $h_n$ maps are injective. It is trivial to see that the $h_n$ maps are surjective.
\par 

Now we are left with the proof that these $h_n$ maps are quandle isomorphisms. Let $P, Q \in M_{1, \lambda^n}$. Then there exist unique $A, B \in M_{1, \lambda}$ such that $A^n=P$ and $B^n=Q$. Now 
\begin{align*}
h_n(P * Q)&= h_n(A^n*B^n)\\
&=h_n(B^{-n} A^n B^n)\\
&=h_n\big((B^{-n} A B^n)^n\big)\\
&=\big(h_1(B^{-n} A B^n)\big)^n\\
&=\big(h_1(A *^n B)\big)^n\\
&=\big(h_1(A)*^n h_1(B)\big)^n\\
&=\big(h_1(B)^{-n}h_1(A)h_1(B)^n\big)^n\\
&=h_1(B)^{-n}h_1(A)^nh_1(B)^n\\
&=h_n(B^{n})^{-1} h_n(A^n) h_n(B^n)\\
&=h_n(A^n)*h_n(B^n)\\
&=h_n(P)*h_n(Q).
\end{align*}
This completes the proof.
\end{proof}

\begin{theorem}\label{thm:6.3}
Let the quandle $M_{\lambda_1, \lambda_2}$ be isomorphic to $M_{\lambda_1', \lambda_2'}$, where $\lambda_2/\lambda_1$ is not a root of unity. Then $M_{\lambda_1^n, \lambda_2^n}$ is isomorphic to $M_{\lambda_1'^n, \lambda_2'^n}$, for all $n \in \mathbb{Z}^+$.
\end{theorem}
\begin{proof}
The proof follows from Theorem \ref{f} and Lemma \ref{higher-isomorphisms}.
\end{proof}

\begin{corollary}
Let $\lambda_2/\lambda_1$ and $\lambda_2'/\lambda_1'$ not be  roots of unity, and $M_{\lambda_1, \lambda_2}$ is not isomorphic to $M_{\lambda_1', \lambda_2'}$. Then $M_{\lambda_1^{1/n}, \lambda_2^{1/n}}$ is not isomorphic to $ M_{\lambda_1'^{1/n}, \lambda_2'^{1/n}},$ for any $n \in \mathbb{Z}^+$.
\end{corollary} 

\begin{proposition}
For $\lambda_1, \lambda_2 \in \mathbb{C}^*$ and $\lambda_1 \neq \lambda_2$, the maximal trivial subquandle in $M_{\lambda_1, \lambda_2}$ is of cardinality $2$.
\end{proposition}
\begin{proof}
Clearly the set $\left\{ \begin{pmatrix}
\lambda_1 & 0 \\ 0 & \lambda_2
\end{pmatrix}, \begin{pmatrix}
\lambda_2 & 0 \\ 0 & \lambda_1
\end{pmatrix} \right\}$ is a trivial subquandle of $M_{\lambda_1, \lambda_2}$ of cardinality $2$. Now suppose $T$ is a trivial subquandle of $M_{\lambda_1, \lambda_2}$, having more than two elements. Since $T$ is trivial, thus the elements of $T$ commutes pair wise which further implies that they are simultaneous diagonalizable (see \cite[Page 569]{MR1878556}), which is a contradiction.
\end{proof}

Since the set of upper triangular matrices in $M_{\lambda}$ is a  commutative set, we have the following result.
\begin{proposition}
The maximal cardinality of a trivial quandle in $M_{\lambda}$ is equal to the cardinality of $\mathbb{R}$.
\end{proposition}

For $\lambda_1\neq \lambda_2 \in \mathbb{C}^*$, $D(\lambda_1, \lambda_2) *D(\lambda_2, \lambda_1) * D(\lambda_1, \lambda_2)=D(\lambda_1, \lambda_2)$ in $M_{\lambda_1, \lambda_2}$.
\begin{lemma}\label{connecting-by-non-commutative}
Let $\lambda_1, \lambda_2 \in \mathbb{C}^*$, and $\lambda_1 \neq \lambda_2$, $\lambda_2\neq \pm 1$, $\lambda_1 \neq -\lambda_2$, $\lambda_1 \neq \lambda_2^2$. Then there exist $A, B \in M_{\lambda_1, \lambda_2}$ such that $AB \neq BA$ and $D(\lambda_1, \lambda_2) * A *B =D(\lambda_1, \lambda_2)$.
\end{lemma}
\begin{proof}
Let $A=\begin{pmatrix}
a & b \\ c & d
\end{pmatrix}, B=\begin{pmatrix}
e & f \\ g & h
\end{pmatrix}
\in M_{\lambda_1, \lambda_2}$, such that $D(\lambda_1, \lambda_2) *A *B=D(\lambda_1, \lambda_2)$. As $\lambda_1 \neq \lambda_2$, thus $AB$ is a diagonal matrix. Now for $\lambda \neq \mu \in \mathbb{C}^*$, $BA=D(\lambda, \mu)$
\begin{align}
&\iff 
\begin{pmatrix}
e & f \\ g & h
\end{pmatrix}
\begin{pmatrix}
a & b \\ c & d
\end{pmatrix}=\begin{pmatrix}
\lambda & 0 \\ 0 & \mu
\end{pmatrix}	\notag \\
&\iff \begin{pmatrix}
a & b \\ c & d
\end{pmatrix}=
\begin{pmatrix}
\lambda & 0 \\ 0 & \mu
\end{pmatrix}
\begin{pmatrix}
h & -f \\-g & e
\end{pmatrix}\frac{1}{\lambda_1 \lambda_2}	\notag \\
&\iff \begin{dcases}
a= \frac{h \lambda}{\lambda_1 \lambda_2}\\
b = \frac{-f \lambda}{\lambda_1 \lambda_2}\\
c =\frac{-g \mu}{\lambda_1 \lambda_2}\\
d=\frac{e \mu}{\lambda_1 \lambda_2}
\end{dcases}\label{eq:6.5}
\end{align}
Now onwards, assume $BA=D(\lambda, \mu)$. Since $A, B\in M_{\lambda_1, \lambda_2}$, $BA=D(\lambda, \mu)$
the determinant argument implies

\begin{align}
&\lambda_1 \lambda_2= ad-bc=\frac{\mu \lambda h e- \mu \lambda f g}{(\lambda_1 \lambda_2)^2}	\notag\\
\implies &\mu \lambda = (\lambda_1 \lambda_2)^2
\end{align}
and the trace argument implies 

\begin{align}
&\lambda_1 + \lambda_2 = a+ d = \frac{\lambda h + \mu e}{\lambda_1 \lambda_2} \notag\\
\implies & \lambda h + \mu e= \lambda_1 \lambda_2 (\lambda_1 + \lambda_2) \notag \\
\implies &\lambda h + \mu (\lambda_1 + \lambda_2 -h)= \lambda_1 \lambda_2 (\lambda_1 + \lambda_2) \notag\\
\implies &\begin{dcases}
h=\frac{(\lambda_1 \lambda_2 -\mu) (\lambda_1 + \lambda_2)}{\lambda- \mu} ~{\textrm{ (here we used } \lambda \neq \mu)}\\
e=\frac{(\lambda_1 + \lambda_2)(\lambda_1 \lambda_2-\lambda)}{\mu -\lambda}\\
fg=-\frac{(\lambda_1+ \lambda_2)^2(\lambda_1 \lambda_2-\mu)(\lambda_1 \lambda_2 -\lambda)+\lambda_1\lambda_2(\mu-\lambda)^2}{(\mu-\lambda)^2}
\end{dcases}
\end{align}
Now onwards take $\mu=\lambda_1$ and $\lambda=\lambda_1\lambda_2^2$ {( $\lambda_2 \neq \pm 1$ is used here, as $\mu \neq \lambda$)} . Thus 
\begin{align}
\begin{dcases}
e=\frac{(\lambda_1+\lambda_2)\lambda_2}{1+ \lambda_2}\\
h=\frac{\lambda_1+\lambda_2}{1+\lambda_2}\\
a=e\\
b=-f \lambda_2\\
c=-\frac{g}{\lambda_2}\\
d=h
\end{dcases} \label{6.8}
\end{align}
and 
\begin{align}
fg=-\lambda_2\frac{-(\lambda_1+\lambda_2)^2 + \lambda_1 (1+ \lambda_2)^2}{(1+ \lambda_2)^2} \label{6.9} ~~
\end{align}
Thus for suitable values of $f$ and $g$, there exist $A, B \in M_{\lambda_1, \lambda_2}$ such that $D(\lambda_1, \lambda_2) * A * B=D(\lambda_1 , \lambda_2)$. We ar now left to prove that $AB\neq BA$.
Now
\begin{align*}
&AB=BA\\
\iff & \begin{pmatrix}
a & b \\ c& d
\end{pmatrix}
\begin{pmatrix}
a & f \\ g & d
\end{pmatrix}=
\begin{pmatrix}
a & f \\ g & d
\end{pmatrix}
\begin{pmatrix}
a & b \\ c & d
\end{pmatrix}\\
\iff &\begin{dcases}
bg=fc\\
af+ bd = ab+ fd\\
ac + dg = ag + cd
\end{dcases}
\end{align*}
Since $\lambda_1 \neq -\lambda_2$ and  $\lambda_2 \neq 0, \pm 1$, the above equations implies that $fg=0$. Now  
\begin{align*}
&fg=0\\
\iff &-(\lambda_1+\lambda_2)^2 + \lambda_1 (1+ \lambda_2)^2=0 {~~ ( \textrm{using Equation\eqref{6.9})}}\\
\iff &\textrm{ either } {\lambda_1=1} \textrm{ or } \lambda_1=\lambda_2^2.
\end{align*}
Since it is given that neither $\lambda_1=1$ nor $\lambda_1 =\lambda_2^2$, we have $fg\neq 0$. This further implies that for the values in { Equation \eqref{6.8}}, we have $AB\neq BA$ and $D(\lambda_1, \lambda_2) * A* B=D(\lambda_1, \lambda_2).$
\end{proof}

\begin{theorem}\label{thm:connecting-by-non-commutative}
Let $\lambda_1, \lambda_2 \in \mathbb{C}^*$ such that $\lambda_1\neq \pm \lambda_2$. Then there exist $A, B \in M_{\lambda_1, \lambda_2}$ such that $AB\neq BA$ and $D(\lambda_1, \lambda_2) * A * B=D(\lambda_1, \lambda_2)$.
\end{theorem}
\begin{proof}
By Lemma \ref{connecting-by-non-commutative}, we need to establish the statement only for the following pairs of eigen values:
\begin{enumerate}
\item for $(\pm 1, \lambda_2)$, where $\lambda_2 \neq \pm 1$,
\item for $(\lambda_1, \pm 1)$, where $\lambda_1 \neq \pm 1$,
\item for $\lambda_1=\lambda_2^2$, where $\lambda_2 \neq \pm 1$.
\end{enumerate}

Here we are proving the statement for the above cases:
\begin{itemize}
\item [1)] For $(\pm 1, \lambda_2)$, where $\lambda_2 \neq \pm 1:$ Pick $k \in \mathbb{C}^*\setminus\{\pm 1, 1/\lambda_2^2\}$. Then by Lemma \ref{connecting-by-non-commutative}, the statement holds for $M_{\pm k, k \lambda_2}$, thus by Theorem \ref{f}, the statement holds for $M_{\pm 1, \lambda_2}$.
\item [2)] For $(\lambda_1, \pm 1)$, where $\lambda_1 \neq \pm 1:$ Pick $k \in \mathbb{C}^*\setminus\{\pm 1, 1/\lambda_1, \lambda_1\}$. Then by Lemma \ref{connecting-by-non-commutative}, the statement holds for $M_{\pm k \lambda_1, \pm k}$, thus by Theorem \ref{f}, the statement holds for $M_{\lambda_1, \pm 1}$.
\item [3)] For $\lambda_1=\lambda_2^2$, where $\lambda_2 \neq \pm 1$: Pick $k \in \mathbb{C}^*\setminus\{\pm 1, 1/(\lambda_2)^2\}$. Then by Lemma \ref{connecting-by-non-commutative}, the statement holds for $M_{ k \lambda_2^2, k \lambda_2}$, thus by Theorem \ref{f}, the statement holds for $M_{\lambda_2^2, \lambda_2}$.
\end{itemize}
\end{proof}

\begin{theorem}\label{prop:noR3}
Let $\R_3$ be the dihedral quandle of order $3$. Then
\begin{enumerate}
\item $\R_3$ is not a subquandle of $M_{\lambda_1, \lambda_2}$ for $\lambda_1\neq \pm\lambda_2 \in \mathbb{C}^*$;
\item $\R_3$ is not a subquandle of $M_{\lambda}$ for $\lambda\in \mathbb{C}^*$.
\end{enumerate}
\end{theorem}
\begin{proof}
We will first prove {\it(1)}. Given $\lambda_1\neq \pm\lambda_2 \in \mathbb{C}^*$, we will consider, without loss of generality, an element $\begin{pmatrix}
a & b\\
c & d
\end{pmatrix}\in M_{\lambda_1, \lambda_2}$, and together with $\begin{pmatrix}
\lambda_1 & 0\\
0 & \lambda_2
\end{pmatrix}$ and $\begin{pmatrix}
a & b\\
c & d
\end{pmatrix} * \begin{pmatrix}
\lambda_1 & 0\\
0 & \lambda_2
\end{pmatrix}$, we will force these three elements to constitute the underlying set of an isomorph of $\R_3$. We note that, for $\R_3$, when two distinct elements are operated the result is the other element. Note that

\begin{align*}
&\begin{pmatrix}
a & b\\
c & d
\end{pmatrix} * \begin{pmatrix}
\lambda_1 & 0\\
0 & \lambda_2
\end{pmatrix} = \begin{pmatrix}
a & b\lambda_2/\lambda_1\\
c\lambda_1/\lambda_2 & d
\end{pmatrix}
\end{align*}
We now force the equality:
\begin{alignat*}{2}
& \begin{pmatrix}
a & b\\
c & d
\end{pmatrix}* \begin{pmatrix}
a & b\lambda_2/\lambda_1\\
c\lambda_1/\lambda_2 & d
\end{pmatrix} &&= \begin{pmatrix}\lambda_1 & 0\\
0 & \lambda_2
\end{pmatrix} \numberthis \label{6.0.10}\\
\iff &\begin{pmatrix}
a & b\\
c & d
\end{pmatrix} \begin{pmatrix}
a & b\lambda_2/\lambda_1\\
c\lambda_1/\lambda_2 & d
\end{pmatrix} &&= \begin{pmatrix}
a & b\lambda_2/\lambda_1\\
c\lambda_1/\lambda_2 & d
\end{pmatrix}\begin{pmatrix}\lambda_1 & 0\\
0 & \lambda_2
\end{pmatrix}\\
\iff & \begin{pmatrix}
a^2+bc\lambda_1/\lambda_2 & b(a\lambda_2/\lambda_1+d)\\
c(a+d\lambda_1/\lambda_2) & bc\lambda_2/\lambda_1+d^2
\end{pmatrix} &&= \begin{pmatrix}
a\lambda_1 & b\lambda_2^2/\lambda_1\\
c\lambda_1^2/\lambda_2 & d\lambda_2
\end{pmatrix}\\
\iff &\begin{cases}
a^2+bc\lambda_1/\lambda_2 = a\lambda_1 \\
b(a\lambda_2/\lambda_1+d) = b\lambda_2^2/\lambda_1\\
c(a+d\lambda_1/\lambda_2)= c\lambda_1^2/\lambda_2 \\
bc\lambda_2/\lambda_1+d^2 =  d\lambda_2
\end{cases}\\
\end{alignat*}

if and only if \begin{numcases}
{} a^2+bc\lambda_1/\lambda_2 = a\lambda_1 \label{6.10} \\
b=0 \, \lor\, a\lambda_2/\lambda_1+d = \lambda_2^2/\lambda_1 \label{6.11}\\
c=0 \, \lor \, a+d\lambda_1/\lambda_2 = \lambda_1^2/\lambda_2 \label{6.12}\\
bc\lambda_2/\lambda_1+d^2 =  d\lambda_2 \label{6.13}
\end{numcases}

We now explore the consequences of $b=0$ or $c=0$. In Equations \eqref{6.10} and \eqref{6.13}, we now have:
\begin{align*}
&\begin{cases}
a^2 = a\lambda_1 \\
d^2 =  d\lambda_2
\end{cases} \iff \begin{cases}
a = 0\,\lor\,   a=\lambda_1 \\
d=0\,\lor\,   d=\lambda_2
\end{cases} \\
\iff & [a=0 \,\land\, d=0] \,\lor\,  [a=0 \,\land\, d=\lambda_2] \,\lor\,  [a=\lambda_1 \,\land\, d=0] \,\lor\, [a=\lambda_1 \,\land\, d=\lambda_2]
\end{align*}
The first three instances, along with $bc=0$, imply $$\lambda_1 \lambda_2 = \det \begin{pmatrix}
a & b\\
c & d
\end{pmatrix} = 0$$ which is impossible since $\lambda_1, \lambda_2\in \mathbb{C}^*$. The instance $a=\lambda_1 \,\land\, d=\lambda_2$ gives rise to the following three instances:
\begin{enumerate}[1.]
\setlength\itemsep{1em}
\item $\begin{pmatrix}
a & b\\
c & d
\end{pmatrix} = \begin{pmatrix}
\lambda_1 & 0\\
0 & \lambda_2
\end{pmatrix}$ which implies that the prospective isomorph to $\R_3$ has at most two elements and therefore cannot be an isomorph to $\R_3$.
\item $\begin{pmatrix}
a & b\\
c & d
\end{pmatrix} = \begin{pmatrix}
\lambda_1 & b\\
0 & \lambda_2
\end{pmatrix}$ with $b\neq 0$ but by performing the operations with the other two elements we do not obtain a quandle structure with these three elements. Namely, forcing
\begin{align*}
&\begin{pmatrix}
\lambda_1 & b\lambda_2/\lambda_1\\
0 & \lambda_2
\end{pmatrix} = \begin{pmatrix}
\lambda_1 & 0\\
0 & \lambda_2
\end{pmatrix} * \begin{pmatrix}
\lambda_1 & b\\
0 & \lambda_2
\end{pmatrix}
\end{align*} implies $$\lambda_2=\lambda_1-\lambda_2 \qquad \iff \lambda_1=2\lambda_2 .$$
Forcing
\begin{align*}
&\begin{pmatrix}
\lambda_1 & b\\
0 & \lambda_2
\end{pmatrix} = \begin{pmatrix}
\lambda_1 & 0\\
0 & \lambda_2
\end{pmatrix} * \begin{pmatrix}
\lambda_1 & b\lambda_2/\lambda_1\\
0 & \lambda_2
\end{pmatrix}
\end{align*}
implies

$\lambda_1^2 + \lambda_2^2=\lambda_1\lambda_2$.
Now using $\lambda_1=2 \lambda_2$, we get $4=1$.
which is absurd.

\item $\begin{pmatrix}
a & b\\
c & d
\end{pmatrix} = \begin{pmatrix}
\lambda_1 & 0\\
c & \lambda_2
\end{pmatrix}$ with $c\neq 0$ with the same remark as for the previous item, arguing along the same lines.
\end{enumerate}

We resume the study of Equations \eqref{6.10}, \eqref{6.11}, \eqref{6.12} and \eqref{6.13}, now certain that $b\neq 0$ and $c\neq 0$:

\begin{align*}
(\dagger\dagger)\begin{cases}
a^2+bc\lambda_1/\lambda_2 = a\lambda_1 \\
a\lambda_2/\lambda_1+d = \lambda_2^2/\lambda_1\\
a+d\lambda_1/\lambda_2 = \lambda_1^2/\lambda_2 \\
bc\lambda_2/\lambda_1+d^2 =  d\lambda_2
\end{cases} \Longrightarrow \begin{cases}
a\lambda_2+d\lambda_1 = \lambda_2^2\\
a\lambda_2+d\lambda_1 = \lambda_1^2 \\
\end{cases} \Longrightarrow\,  \lambda_2^2 = \lambda_1^2 \, \Longrightarrow\,  \lambda_2 = \pm\lambda_1
\end{align*}
but $\lambda_2 = \pm\lambda_1$ is impossible by hypothesis. This concludes the proof of {\it (1)}.

\par

We now prove {\it (2)}. Since $M_1$ is isomorphic to $M_\lambda$ for any $\lambda\in \mathbb{C}^*$ (see Corollary \ref{f.1}), we will work with $M_1$. Given $\begin{pmatrix}
a & b\\
c& d
\end{pmatrix}$, without loss of generality, we will impose an $\R_3$ quandle structure on the elements $$\begin{pmatrix}
1 & 1\\
0 & 1
\end{pmatrix}, \begin{pmatrix}
a & b\\
c& d
\end{pmatrix}, \begin{pmatrix}
a & b\\
c& d
\end{pmatrix}* \begin{pmatrix}
1 & 1\\
0 & 1
\end{pmatrix} = \begin{pmatrix}
a-c & a-c+b-d\\
c & c+d
\end{pmatrix}$$
So, we impose
\begin{align*}
&\begin{pmatrix}
a-c & a-c+b-d\\
c & c+d
\end{pmatrix}* \begin{pmatrix}
1 & 1\\
0 & 1
\end{pmatrix}=\begin{pmatrix}
a & b\\
c& d
\end{pmatrix} \\
\iff &\begin{pmatrix}
a-c & a-c+b-d\\
c & c+d
\end{pmatrix} \begin{pmatrix}
1 & 1\\
0 & 1
\end{pmatrix} = \begin{pmatrix}
1 & 1\\
0 & 1
\end{pmatrix}\begin{pmatrix}
a & b\\
c& d
\end{pmatrix} \\
\iff &\begin{pmatrix}
a-c & 2(a-c)+b-d\\
c & 2c+d
\end{pmatrix} =\begin{pmatrix}
a+c & b+d\\
c& d
\end{pmatrix} \text{ so $c=0$ }\\
\iff &\begin{pmatrix}
a & 2a+b-d\\
0 & d
\end{pmatrix} =\begin{pmatrix}
a & b+d\\
0& d
\end{pmatrix} \text{ so $a=d$ }
\end{align*}
and since the trace of the matrices equals $2$, then $a=1=d$ so
\begin{align*}
&\begin{pmatrix}
a & b\\
c& d
\end{pmatrix} = \begin{pmatrix}
1 & b\\
0& 1
\end{pmatrix}=\begin{pmatrix}
a-c & a-c+b-d\\
c& c+d
\end{pmatrix}
\end{align*}
so altogether we  have less than three elements and therefore we do not have an isomorph with $\R_3$. This completes the proof.

\end{proof}

\begin{theorem}\label{thm:6.11}
The dihedral quandle of order $3$, $\R_3$, is a subquandle of $M_{\lambda, -\lambda}$, for any $\lambda\in \mathbb{C}^*$.
\end{theorem}
\begin{proof}
Since $M_{\lambda, -\lambda}$ is isomorphic with $M_{1, -1}$, we will use the latter for our proof. Without loss of generality, we will prove that $$\left \{\, \begin{pmatrix}
a & b\\
c & d
\end{pmatrix}, \begin{pmatrix}
1 & 0\\
0 & -1
\end{pmatrix}, \begin{pmatrix}
a & b\\
c & d
\end{pmatrix} \ast \begin{pmatrix}
1 & 0\\
0 & -1
\end{pmatrix}\, \right\} = \left\{\, \begin{pmatrix}
a & b\\
c & d
\end{pmatrix}, \begin{pmatrix}
1 & 0\\
0 & -1
\end{pmatrix}, \begin{pmatrix}
a & -b\\
-c & d
\end{pmatrix} \, \right \}  \, \subset\,   M_{1, -1} ,$$ is the underlying set of $\R_3$ in $M_{1, -1}$, find conditions for $a, b, c, d$ and prove that the subquandle so generated is in fact $\R_3$. We thus resume the study of the system $(\dagger\dagger)$ from the proof of Proposition \ref{prop:noR3}, with $\lambda_1=1=-\lambda_2$:

\begin{align*}
(\dagger\dagger\dagger)&\begin{cases}
a^2-bc = a \\
-a+d = 1\\
a-d=-1 \\
-bc+d^2 =  -d
\end{cases} \Longleftrightarrow \begin{cases}
d-a=1 \\
d^2 +d-bc=0\\
a^2-a-bc =0
\end{cases} \Longrightarrow \begin{cases}
d-a=1 \\
d_{\pm}={\displaystyle\frac{-1\pm\sqrt{1+4bc}}{2}}\\
a_{\pm}={\displaystyle\frac{+1\pm\sqrt{1+4bc}}{2}}
\end{cases} \Longrightarrow \\
&\Longrightarrow \begin{cases}
d_{+}-a_{-}=1 \\
d_{+}={\displaystyle\frac{-1+\sqrt{1+4bc}}{2}}\\
a_{-}={\displaystyle\frac{+1-\sqrt{1+4bc}}{2}} \quad \text{ (along with $bc=3/4$)}\\
0=a_{-}+d_{+} \qquad (\text{trace equation})\\
-1=a_{-}d_{+}-bc \quad  (\text{determinant equation})
\end{cases}  \Longrightarrow \begin{cases}
a_{-}=-1/2\\
d_{+}=1/2 \\
bc=3/4
\end{cases}
\end{align*}
so the three elements of the prospective isomorph with $\R_3$ are: $$\begin{pmatrix}
1 & 0\\
0 & -1
\end{pmatrix}, \begin{pmatrix}
-1/2 & b\\
c & 1/2
\end{pmatrix}, \begin{pmatrix}
-1/2 & -b\\
-c & 1/2
\end{pmatrix}\qquad \text{with} \qquad bc=3/4 .$$
We will now check that whenever two distinct elements from these three are operated, the result is the other element.
\begin{align*}
&\begin{pmatrix}
1 & 0\\
0 & -1
\end{pmatrix}\ast \begin{pmatrix}
-1/2 & \pm b\\
\pm c & 1/2
\end{pmatrix}= \begin{pmatrix}
-1/2 & \pm b\\
\pm c & 1/2
\end{pmatrix}\begin{pmatrix}
1 & 0\\
0 & -1
\end{pmatrix}\begin{pmatrix}
-1/2 & \pm b\\
\pm c & 1/2
\end{pmatrix}=\\
&= \begin{pmatrix}
-1/2 & \mp b\\
\pm c & -1/2
\end{pmatrix}\begin{pmatrix}
-1/2 & \pm b\\
\pm c & 1/2
\end{pmatrix}=\begin{pmatrix}
1/4 - bc & \mp b\\
\mp c & bc-1/4
\end{pmatrix}=\begin{pmatrix}
-1/2 & \mp b\\
\mp c & 1/2
\end{pmatrix}
\end{align*}

\par

\begin{align*}
&\begin{pmatrix}
-1/2 & \pm b\\
\pm c & 1/2
\end{pmatrix}\ast \begin{pmatrix}
1 & 0\\
0 & -1
\end{pmatrix}= \begin{pmatrix}
1 & 0\\
0 & -1
\end{pmatrix}\begin{pmatrix}
-1/2 & \pm b\\
\pm c & 1/2
\end{pmatrix}\begin{pmatrix}
1 & 0\\
0 & -1
\end{pmatrix}=\\
&= \begin{pmatrix}
-1/2 & \pm b\\
\mp c & -1/2
\end{pmatrix}\begin{pmatrix}
1 & 0\\
0 & -1
\end{pmatrix}=\begin{pmatrix}
-1/2 & \mp b\\
\mp c & 1/2
\end{pmatrix}
\end{align*}

\par

\begin{align*}
&\begin{pmatrix}
-1/2 & \pm b\\
\pm c & 1/2
\end{pmatrix}\ast \begin{pmatrix}
-1/2 & \mp b\\
\mp c & 1/2
\end{pmatrix}= \begin{pmatrix}
-1/2 & \mp b\\
\mp c & 1/2
\end{pmatrix}\begin{pmatrix}
-1/2 & \pm b\\
\pm c & 1/2
\end{pmatrix}\begin{pmatrix}
-1/2 & \mp b\\
\mp c & 1/2
\end{pmatrix}=\\
&= \begin{pmatrix}
1/4-bc & \mp b\\
\pm c & 1/4-bc
\end{pmatrix}\begin{pmatrix}
-1/2 & \mp b\\
\mp c & 1/2
\end{pmatrix}=\begin{pmatrix}
-1/2 & \mp b\\
\pm c & -1/2
\end{pmatrix}\begin{pmatrix}
-1/2 & \mp b\\
\mp c & 1/2
\end{pmatrix}=\\
&=\begin{pmatrix}
1/4+bc & 0\\
0 & -bc-1/4
\end{pmatrix}=\begin{pmatrix}
1 & 0\\
0 & -1
\end{pmatrix}
\end{align*}

This completes the proof.
\end{proof}

\section{Future Work}\label{sec:future_work}
In this section we collect some questions for future work.

\begin{conjecture}
Let $Q$ be an infinite type quandle. Then $\mathcal{Q}_n(Q)$ is not isomorphic to $\mathcal{Q}_m(Q) $, for $n \neq m$ and  $n, m\in \mathbb{Z}^+$.
\end{conjecture}

\begin{question}
Let $\omega_1$ and $\omega_2$ be distinct $n$-th primitive roots of unity. Is there any relation between the quandles $M_{1, \omega_1}$ and $M_{1, \omega_2}$?
\end{question}

\begin{question}
Does there exist a connected quandle $Q$ such that $\mathcal{Q}_n(Q)$ is connected for all $n \in \mathbb{Z}^+$?
\end{question}

\begin{question}
Classify infinite type $M_{\lambda_1, \lambda_2}$ quandles, where $\lambda_1, \lambda_2 \in \mathbb{C}^*$.
\end{question}

\begin{question}
Can a dihedral quandle $\R_{2k+1}$, where $k \in \mathbb{Z}^+$, be subquandle of $\M_{\lambda_1, \lambda_2}$ or $\M_{\lambda}$, where $\lambda, \lambda_1, \lambda_2 \in \mathbb{C}^*$ and $\lambda_1 \neq \lambda_2 $?
\end{question}

\begin{question}
What other well known connected quandles are subquandles of $\Conj(\GL(2, \mathbb{C})$?
\end{question}

\begin{ack}Pedro Lopes is supported by CAMGSD via project UIDB/04459/2020, financed by national funds FCT/MCTES (PIDDAC). Manpreet Singh is supported by UIDP/04459/2020 Post-Doctoral Research Fellowship at CAMGSD, financed by national funds FCT/MCTES (PIDDAC).
\end{ack}

\bibliography{references1}{}
\bibliographystyle{abbrv}

\end{document}